\documentclass{article}
\usepackage{amsmath, amsthm, amsfonts, amssymb, verbatim, amscd, graphicx, hyperref}

\theoremstyle{plain}
\newtheorem{tpcl}{Tpcl}[section]

\newtheorem{lemma}[tpcl]{Lemma} 
\newtheorem{theorem}[tpcl]{Theorem} 
\newtheorem{proposition}[tpcl]{Proposition} 
\newtheorem{corollary}[tpcl]{Corollary} 

\theoremstyle{definition}
\newtheorem{definition}[tpcl]{Definition} 
\newtheorem{example}[tpcl]{Example} 
\newtheorem{examples}[tpcl]{Examples} 
\newtheorem{remark}[tpcl]{Remark}

\DeclareMathOperator{\Coin}{Coin}

\DeclareMathOperator{\lcm}{lcm}
\DeclareMathOperator{\ind}{ind}
\DeclareMathOperator{\Ind}{Ind}
\DeclareMathOperator{\im}{Im}
\DeclareMathOperator{\id}{id}

\newcommand{\Reid}{\mathcal R}
\newcommand{\R}{\mathbb R}
\newcommand{\Z}{\mathbb Z}

\title{A Nielsen theory for coincidences of iterates}
\author{Philip R. Heath\thanks{Department of Mathematics, Memorial University, St. John's, NL, Canada A1C 5S7} \and Christopher Staecker\thanks{Department of Mathematics and Computer Science, Fairfield University, Fairfield, CT, USA 06824}}

\begin{document}

\maketitle

\begin{abstract} 
As the title suggests, this paper gives a Nielsen theory of coincidences of iterates of two self maps 
$f, g : X \to X$ of a closed manifold $X$. The ideas is, as much as possible, to generalize Nielsen type periodic point theory, but there are many obstacles. Many times we get similar results to the 
 \lq \lq classical ones" 
in Nielsen periodic point theory, but with stronger hypotheses. 
\end{abstract} 
\section{Introduction} This paper seeks to generalize Nielsen periodic point
theory to iterates of coincidences. There is much written in the literature concerning 
both Nielsen coincidence theory, and Nielsen periodic point theory. 
Let $f, g : X \to Y$ be maps (continuous functions) of closed manifolds 
$X$ and $Y$ of the same dimension. We use the symbol
$\Phi(f,g)$ to 
denote the set of coincidences of $f$ and $g$, that is $\Phi(f,g)= \{ x \in X\mid f(x) = g(x) \}$. 
The aim of Nielsen coincidence theory is to define a lower bound (as sharp as possible), for the set $M\Phi(f,g)$,  which is defined to be 
$\min(\#\{\Phi(f_1, g_1)\mid f_1 \simeq f, g_1 \simeq g\})$, where $\simeq$ denotes homotopy, and $\#$ cardinality. In a similar setting for a fixed positive integer $n$, 
Jiang (\cite{bojujiang}) introduced two Nielsen type numbers $NP_n(f)$ and $N\Phi_n(f)$ where
$f$ is a self map of $X$. These two numbers are
homotopy invariant lower bounds respectively for the number 
$MP_n(f)$, which is the cardinality
of the smallest among the sets $P_n(f_1) = \Phi(f_1^n) - \bigcup_{m|n \ m\neq n}\Phi(f_1^m),$ 
as $f_1$ ranges over all maps homotopic to $f$, and of 
$M\Phi_n(f)=\min\{\#\Phi(f_1^n)\mid f_1\simeq f \}$ where 
$\Phi(f_1^n)= \{ x \in X\mid f_1^n(x) = x \}$
is the fixed point set of $f_1^n$. It is important in the definitions of $MP_n(f)$ and
$M\Phi_n(f)$, 
to note that we only allow homotopies of $f$,  not of $f^n$. 

In this paper then, we define 
two Nielsen type numbers $NP_n(f,g)$ and $N\Phi_n(f, g)$ (in the context that $X=Y$), 
both of which are homotopy invariant lower bounds for the cardinalities of 
appropriate sets. The first of these $MP_{n}(f,g)$, of which $NP_n(f,g)$ is a lower bound,
is a straight forward generalization of $MP_n(f)$. In fact 
$$ MP_{n}(f,g) = \min \# \{P_{n}(f_1,g_1) | f_1 \simeq f \mbox{ and } g_1 \simeq g\},$$ where  $P_{n}(f,g)$ denotes the set of points $x$ with $f^n(x) = g^n(x)$ but $f^m(x) \neq g^m(x)$
 for any $m\mid n$. We will often write that $x$ is a coincidence \emph{at level $n$} when $f^n(x) = g^n(x)$.

Before we discuss the number of which $N\Phi_n(f, g)$ is a lower bound, we ask the reader 
to observe the following equalities which hold automatically in periodic point theory
$$\Phi (f^n)= \bigcup_{m \mid n} P_m(f) = \bigcup_{m \mid n} \Phi(f^m).$$
Note also that the $P_m(f)$ are 
disjoint. When it comes
to coincidences of iterates, as we will see, the above equalities that we take for granted 
in periodic point theory do not universally hold true. 
We will show for coincidences of iterates, 
that there are examples that variously
illustrate the following possibilities:
\begin{equation}\label{unions}
\Phi (f^n, g^n)\neq \bigcup_{m \mid n} \Phi(f^m, g^m) =\bigcup_{m \mid n} P_m(f, g)
\neq \bigsqcup_{m \mid n} P_m(f, g), 
\end{equation}
where $\bigsqcup$ denotes disjoint union. We will see immediately 
an example where
$\Phi (f^n, g^n)$ and $\bigcup_{m \mid n} \Phi(f^m, g^m)$ can be very different. 
However it is not until
Section \ref{nphisection} that we will see an example where $\bigcup_{m \mid n} P_m(f, g)
\neq \bigsqcup_{m \mid n} P_m(f, g) $. To put this last statement
another way, we are saying  for
coincidences of iterates,
that the $P_m$ may not be disjoint. 
The middle equality of the display \eqref{unions} is of course, always true. 

The differences shown in the numbered display above 
give us three possibilities for defining the minimum number
$M\Phi_n(f,g)$ of which $N\Phi_n(f,g)$ is to be a homotopy invariant
a lower bound. 
Perhaps counter intuitively we define it to be: 
$$M\Phi_n(f,g):= \min \# \left\{ \bigsqcup_{m|n} P_{m}(f_1,g_1) \mid f_1 \simeq f \mbox{ and } g_1 \simeq g\right\}. $$
We will not discuss the technical reason
why we choose to minimize the disjoint union rather than the ordinary union
until Section \ref{nphisection}. However we feel it is instructive
to give  an example now, 
to show  why we need to  reject the other possible candidate
$$MC_n(f, g) := \min\{\#\Phi(f_1^n, g_1^n) \mid f_1\simeq f \mbox{ and } g_1\simeq g\}, $$ 
for the number  
 $M\Phi_n(f,g)$.

In fact, the example below also allows us to 
introduce the reader to some of the other 
hurdles one has to overcome in order to generalize
Nielsen periodic point theory to our context. It  illustrates a 
key difference in the behaviour 
of coincidences of iterates over that of periodic points. In particular,  if
$x$ is a periodic point of $f$, then so is $f^j(x)$ for any positive integer $j$. On 
the other hand, if $x$ is a 
coincidence of $f$ and $g$, there is no guarantee that it is a coincidence 
of $f^n$ and $g^n$ for for any $n>1$.
\begin{example} \label{mcexample}
Let $f$ be the self map of $S^1$ of degree $-1$, 
and let $g$ be a small rotation: $g(e^{i\theta}) = e^{i(\theta +\epsilon)}$ 
for some small $\epsilon \neq 0$. Since $N(f,g) = 2$, it is easy to see that
the pair $f,g$
has two nonremovable coincidence points at level 1. On the other hand
$\Phi(f^2,g^2) = \emptyset$, that is, 
there are no coincidence points of $f^2$ and $g^2$ at all. 

To see this, note that if $f^2(e^{i\theta}) = g^2(e^{i\theta})$, then since
$ f^2(e^{i\theta}) = e^{i\theta}$, and $g^2(e^{i\theta}) =
e^{i(\theta + 2\epsilon)}$, then $\theta
= \theta + 2 \epsilon + 2k \pi$, which is impossible for small epsilon.
Thus for this example 
$\Phi( f^2, g^2) = \emptyset$, while $ \bigcup_{m \mid 2} \Phi(f^m, g^m)= \Phi(f,g)$ has
two elements. Furthermore
$M\Phi_2(f,g) = 
\min \# \{ P_1(f_1, g_1) \sqcup P_2(f_1, g_1) \mid f_1 \simeq f, g_1 \simeq g \} = 2$ and
$MC_2(f,g) =\min \{ \# \Phi(f_1^2, g_1^2) \mid f_1 \simeq f, g_1 \simeq g \} = 0.$
As we will see when we have defined $N\Phi_n(f,g)$ this example shows that we can have 
$MC_n(f,g) < N\Phi_n(f,g)$.  The point then, is that 
$MC_n(f,g)$ can fail
to account for coincidences at levels other than $n$. 

This example also shows 
(for $\epsilon$  irrational) that the trajectories $\{x, f(x), f^2(x), \dots \}$ and $\{x, g(x), g^2(x), \dots \}$ of a
 coincidence point $x$ can be  infinite. 
\end{example}
At the risk of beating the point to death, the example 
also shows that 
$MC_2(f , 1) \neq M\Phi(f^2)$. This may appear at first sight to contradict the well 
known result of Brooks (\cite{broo71}) that, under mild conditions, 
any variance in the cardinality of a coincidence set can be obtained by varying only one of the maps. The apparent contradiction is resolved when we point out that Brooks' result 
does not apply to iterates of maps when the only homotopies of $f^n$ and $g^n$ that we allow 
come from level 1 homotopies of $f$ and $g$. 

As it turns out, if $f$ and $g$ commute with each other (i.e. $fg= gf$) then the distinctions
in the display \eqref{unions} both disappear. In particular in such cases we have that 
$\Phi( f^m, g^m ) \subset \Phi(f^{qm}, g^{qm})$ for all positive $q$ (see Lemma \ref{commuting1}). However 
restricting to commuting maps and their homotopies would be far too restricive. What 
we do need however (in order even
for our boosting functions to be well defined on 
Reidemeister sets), is that the two induced homomorphisms $f_*$ and $g_*$ commute at the level of the fundamental group $\pi_1(X)$ (see Definition \ref{boostingfunctions}). 
In fact in all our examples the homotopy class of the maps discussed contain representatives 
that commute with each other, and this is more than enough for our purposes. 

The lack of geometric boostings is just one of several roadblocks one needs to 
navigate in the process of generalizing periodic point theory to coincidences 
of iterates. Another  block 
we want to mention here is the fact that orbits don't work. In periodic point theory, 
orbits play an important role in certain examples, but not in fact,
in the vast majority of the spaces we use in our examples here (but see Section \ref{orbitsection}). 
We will discuss this point about 
orbits in more detail in Section \ref{reduciblesection}, but for now let $x\in \Phi(f^n)$ for some map $f$. Then the trajectory
(or orbit) $\{x, f(x), f^2(x), \dots \}$ of $x$ is finite and of length less than $n$. 
However as Example \ref{mcexample} shows, for coincidences, 
the trajectories $\{x, f(x), f^2(x), \dots \}$ and $\{x,g(x),g^2(x),\dots \}$ of a coincidence point $x \in \Phi(f^n,g^n)$
need not be finite, let alone less than or equal to its level $n$ (see also Example
\ref{noorbits2}). 
The net result of all of this is that our theory needs to be defined in terms of
classes, not in terms of depth of orbit, as in periodic point theory (see Definition \ref{NP}). 
For most of the examples we use, this is in fact no disadvantage since the
spaces are the equivalent of being essentially toral (the definition makes no sense for 
coincidences). We refer the reader to a related discussion in Section \ref{orbitsection}. 

We come next to the question of imitating some of the \lq \lq classical" results of
periodic point theory that hold true for all
maps on tori, and more generally on 
nil and certain maps on solvmanifolds (see \cite{hpy, hy, hk1}). These results 
compare $N\Phi_n(f)$ to combinations of the $N(f^m)$ for $m\mid n$, 
and when $N(f^n) \neq 0$, give 
M\"obius inversion type results for $NP_n(f)$. A key property that allows these periodic point 
results to go through is one that is called
essentially reducible to the GCD. In fact all maps on nil and solvmanifolds are essentially
reducible to the CGD (\cite[Corollary 4.12]{hk1}). On the other hand, the analogous property
(essentially coincidence reducible to the GCD) is not even universally 
true on $S^1$ (where the commuting property of
$f_*$ and $g_*$ is automatic). We refer the reader to Example \ref{counterexamplept5}. 

We do however have a new but weaker result, for all maps that 
are \lq \lq coincidence essentially reducible" (Theorem \ref{pseudomob}), and we investigate
conditions for which the full results hold. Our results include a 
characterization of \lq \lq coincidence essentially reducible to the GCD" for maps
on $S^1$ (Theorem \ref{circlegcd}), and show it also holds on all spaces where
(in addition to the 
commutativity of $f_*$ and $g_*$), we have that $g_*$ is 
invertible (Theorem \ref{gcdmatrix}). In particular this is true if the linearization of 
a self map $g$ of a torus is a 
matrix which is invertible over $\Z$. In addition, by giving examples on the Klein 
Bottle, we hint at the generalizations of our findings
to solvmanifolds. We hint,  rather than give full details, since that
is beyond the scope of the paper.

Our theory is of course applicable to the case where $g$ is the constant map, in other words
to roots of iterates. This subject has already been studied by Brown, Jiang and Schirmer in 
\cite{BJS}. One might suspect that many of the results of that paper could be obtained
simply by putting $g$ equal to a constant map in this one. 
This however is not the case. In fact the two 
theories are deeply incompatible. We will explore this in greater detail in Section \ref{rootsection}, 
but for now we wish to mention two things that emphasize this incompatibility. The first is 
that the
root theory in \cite{BJS} is heavily dependent on the various choices of
 the image of the 
constant map \lq \lq $g$". To put this another way, 
 the theory in \cite{BJS} is not homotopy invariant with respect to 
homotopies of $g$. And of course ours is. 
To describe the second incompatibility,
we want to say first, that it should be clear that the concept of
reducibility is foundational to both theories.  However, 
even with respect to 
this very basic concept, the two theories are not the same. To say it in just a few words,
in our work we consider reductions only for  $m\mid n$. In \cite{BJS} they consider 
reductions with $m< n$ but $m \nmid n$. 

We outline the paper as follows. Following this introduction we give a preliminary section
in which among other things, we establish our notation using  
the modified fundamental group approach. We briefly  recall
standard coincidence theory as applied to iterates of self maps.
 We include a short subsection on linearization of maps on tori, illustrating it
 with an example which will be useful later.  Next in Section \ref{iteratesection} 
we discuss the various relationships among the 
iterates,  separating the geometry and the algebra.
In the process we recall 
 some of the basic concepts of Nielsen periodic point theory, and show the
necessary detours and conditions we need to make in order to proceed with our generalization
to coincidences. In Section \ref{npsection}, we define our number $NP_n(f,g)$, give some of its
properties, show where direct generalizations fail,  and introduce conditions under which the 
\lq \lq classical" results can be recovered. In Section \ref{nphisection}, we do the same thing for 
$N\Phi_n(f)$. In Section \ref{invertiblesection} we study coincidence theory of iterates with the additional 
assumption that the map $g$ is invertible. We indicate for essentially toral spaces that with this
 invertibility
condition, our theory and Nielsen periodic point theory for the map $g^{-1}f$ coincide. 
In Section 7 we we discuss the relationship of our theory to
the theory of roots of iterates given in \cite{BJS}.   We close the main part 
of the paper with a 
short section where we discuss two open questions both related to Wecken type considerations. 
The final section of the paper is an appendix 
where we prove on $S^1$, with a pair of maps $f$ and $g$ which can be represented by 
integers $a$ and $b$ respectively, that $f$ and $g$ are coincidence reducible to the gcd 
if and only if $a$ and $b$ are relatively prime.  This property  is foundational to the main 
computational theorems in Sections \ref{npsection} and \ref{nphisection}. 

The authors would like to thank Nathan Jones for helping us 
with the final step in the proof of Theorem \ref{circlegcd}, and Jerzy Jezierski for bringing reference \cite{jg} to our attention.

\section{Preliminaries, Standard Nielsen coincidence theory of iterates of 
self maps}\label{prelim} 
In this section we review standard Nielsen coincidence theory as it applies to iterates of self maps. We use the modified fundamental group approach as 
in \cite{ gh} (see also \cite{hpy,hy}).  In this approach we 
separate the geometry from the algebra, and by assigning an index (or semi-index) to the
Reidemeister classes,  we are able to deal with the possibility of having different
empty classes, the supposed advantage of the covering space approach.  This section then
will establish our notation.  At the end of the section we
remind the reader of the concept of the linearization of a map on  a torus. We will 
later make an 
oblique reference to linearization on solvmanfiolds,  but since we 
consider only the Klein Bottle, we will not go into a lot of detail. 

The literature contains many papers on 
the Nielsen theory of coincidences. In many settings, 
fixed point theory is a special case of coincidence theory where $g$ is taken to be the identity map. At times however, this viewpoint is overly simplistic, and this  is certainly the case 
with  respect to perioidc point theory, and the theory of iterates of coincidences.  There are
places however,  where generalizations  are entirely straight forward. Some of these
appear below.

\subsection{Geometric classes of iterates of self maps.}
Throughout the paper all spaces $X$
will denote closed manifolds, and  
$f, g:X\to X$ will be self maps of $X$. 
In this subsection we remind the reader of standard coincidence theory, and make 
some straight forward applications of it to the iterates $f^n$ and $g^n$ of $f$ and $g$
respectively. 

We say that $x,y\in\Phi(f^n, g^n)$ are 
{\it Nielsen equivalent at level $n$} provided that there is a path 
$c$ from $x$ to $y$ so that (relative end points) $f^n(c)\simeq g^n(c)$.
For $n= 1$ this is the ordinary Nielsen coincidence relation. 
The set of equivalence classes thus generated 
will be denoted by $\Phi(f^n, g^n )/\sim$. We call this the 
set of {\it [geometric] Nielsen classes} for $f^n$ and $g^n$. 

Using either the standard coincidence 
index (see for example \cite{vic}) or the semi-index in case $X$ is
not orientable
(see \cite{dj93,heathfibresurvey}), we may for each 
geometric class ${\bf A}^n \in \Phi(f^n,g^n)/\sim$ associate an integer
denoted $\ind({\bf A}^n)$. 
The classes for which this integer is nonzero are called 
{\it essential} Nielsen classes.
For any positive integer $n$, 
the {\it Nielsen number} $N(f^n, g^n)$ of $f^n$ and $g^n$ is then the 
number of essential classes of $f^n$ and $g^n$. This number is a lower bound for 
$M\Phi(f^n, g^n)$ which should be carefully distinguished from 
$M\Phi_n(f,g)$. 
For $M\Phi(f^n, g^n)$ we consider homotopies that range over all maps $h$ and $k$ that are homotopic respectively to
$f^n$ and $g^n$. Of course this will include homotopies induced by
homotopies
$f\simeq f_1$ and $g \simeq g_1$ of $f$ and $g$ respectively. 
This last kind of homotopy, and this kind only,
is the kind we consider in the main body of the paper. 
When $g$ is the identity and we disallow homotopies of $g$, then the number of essential classes of $f$ denoted $N(f)$, is the ordinary Nielsen number of $f$. 
\subsection{Algebraic and geometric classes and their relationship. }
We come now to the algebraic side of the story. 
In what follows we shall not distinguish between a path and its path 
class 
in the fundamental groupoid $\pi_1(X)$.
Thus $c$ can denote both a path and a path class in $\pi_1(X)$.
In addition 
if $h: X \to X$ is a map, $h(c)$ will denote either a path or class.
If $c$ is a path from $a$ to $b$,
then $c^{-1}$ is the path from $b$ to $a$ defined by $c^{-1}(t) = c(1-t)$.

Choose
a base point $x_0\in X$. For simplicity we work with base point preserving maps
$f, \ g :X \to X$ with respect to this chosen base point. 
We can do this without loss of generality, since in
manifolds base points are always closed and non-degenerate. In particular
each homotopy class
has a representative that is base point preserving with respect to the chosen base point. 

In this way,  for each positive integer $n$
we have induced homomorphisms 
$f_*^n, \ g_*^n: \pi_1(X,x_0) \to \pi_1(X,x_0)$, and these in turn 
determine an equivalence relation on $\pi_1(X,x_0)$ 
(doubly-twisted conjugacy) defined by the rule that
$\alpha\sim\beta$ in $\pi_1(X,x_0)$ if and only if 
there exists $\gamma\in\pi_1(X,x_0)$ with 
$\alpha = g_*^{n}(\gamma)\beta f_*^{n}(\gamma^{-1})$. The
resulting
classes are called Reidemeister classes.
The Reidemeister class containing $\alpha$ is denoted by 
$[\alpha]^n$. 
The set of all Reidemeister classes is denoted by 
${\cal R}(f_*^{n}, g_*^{n})$, and its cardinality is the \emph{Reidemeister number} $R(f^n,g^n)$. There is an exact sequence of based sets,
\[
\pi_{1} (X, x_0) \longrightarrow
\pi_{1}(X, x_0) \stackrel{j_n}{\longrightarrow}
{\cal R}(f^{n}_*, g^{n}_* ) \to 1
\]
where the first function 
takes an element $\alpha$ to 
$g_*^{n}(\alpha) f_*^{n}(\alpha^{-1})$, and the second $j_n$ places an
element 
$\beta$ in its Reidemeister class $[\beta]^n$.
If $\pi_1(X,x_0)$ is abelian there is a canonical group structure on 
${\cal R}(f_*^{n}, g_*^{n})$,
moreover in this case 
the sequence consists of 
groups and homomorphisms.
All the above constructions are independent of the choice of
base point and path classes in the sense that there exists bijections
between the various Reidemeister sets 
(see \cite{gh}). When $\pi_1(X)$ is Abelian we write composition of functions
additively, and we have:
\begin{theorem} \label{abelianexactseq}
\label{1} (Guo Heath \cite{gh, heathfibresurvey} )
Let $f, \ g: X \to X$ be maps with 
$\pi_1(X,x_0)$ Abelian, then the sequence
\[
0 \to \Coin(f_*^{n}, g_*^{n}) \to \pi_1(X,x_0)
\stackrel{ g_*^{n} - f^{n}_*}{\longrightarrow} \pi_1(X,x_0)
\stackrel{j_n}{\to} {\cal R}(f_*^{n}, g_*^{n}))\to 0
\]
is an exact sequence of Abelian groups and homomorphisms, where 
\[ 
 \ \ \ \ \ \ \ \ \ \ \ \   \ \ \ \ \ \ \ \ \ \ \ \
\Coin(f_*^n, g_*^n) = \{ \alpha \in \pi_1(X,x_0) \mid f_*^n(\alpha) = g_*^n(\alpha) \}. 
 \ \ \ \ \ \ \ \ \ \ \ \  \ \ \ \ \ \ \ \ \ \ \ \
{\Box}
\]
\end{theorem} 

The algebraic and geometric components of the theory are related 
by an injective function $$\rho_n=\rho:\Phi(f^n, \ g^n)/\sim \ \to 
{\cal R}(f_*^{n}, g^n_*)$$ 
defined as follows:
Given $x\in{\bf A}^n$ we choose 
a path $c$ from the base point $x_0$ to $x$. We can then 
define $\rho_n({\bf A}^n)=[g^n(c)f^n(c^{-1})]^n$. 
This will be independent of $c$ and of the choice of $x$ within ${\bf A}^n$. 
An algebraic class $[\alpha]^n$ is said to be {\it nonempty} 
if it lies in the image of $\rho_n$. Following the modified fundamental group approach
as in (\cite{gh, heathfibresurvey, hpy, hy}), we next  assign an index (or semi-index for our
Klein Bottle examples)
to the Reidemeister classes.
The index (semi-index) $\Ind([\alpha]^n)$ of a class 
$[\alpha]^n \in {\cal R}(f_*^n, g_*^n)$ 
is defined as follows 
$$\Ind([\alpha]^n) = \left\{ \begin{array}{ll} 
\ind({\bf A}^n) & \mbox{if} \ \ [\alpha]^n = \rho_n({\bf A}^n), \\ 
0 & \mbox{otherwise,} 
\end{array} \right. 
$$ 
where $\ind({\bf A}^n)$ is the integer defined in the geometric section, the usual coincidence index, or the coincidence semi-index of \cite{dj93}.
As with the geometric classes, an algebraic class is {\it essential} provided 
it has nonzero index (semi-index). 
We denote the set of essential algebraic classes by
$ {\cal E}(f^n, g^n ) \subseteq {\cal R}(f_*^n, g_*^n)$. Clearly $N(f^n, g^n) = 
\#({\cal E}(f^n, g^n ))$. 

\subsection{Linearization and weakly Jiang maps.}
Although most of our examples in this paper will be on tori, we
we will be using the Klein Bottle as a kind of representative solvmanifold. We will not 
go into details of linearizations of maps on these spaces, but refer the reader to 
\cite{hk1, hk2, hk3} for the fixed point case, and to \cite{heathfibresurvey} for
some aspects of the coincidence case. Linearization of maps on tori are
very simple. If $f: T^r \to T^r$ is a map of an $r$ torus, then we can identify the linearization 
of $f$ with the induced homomorphism on $\pi_1(T^r) \cong \Z^r$. 
Using the standard basis for $\Z^r$ we can then identify this homomorphism with a matrix $F$. This same matrix can then be used to define a map in the homotopy class of $f$, namely that map 
which is induced from $F:\R^r \to \R^r$ defined by matrix multiplication on 
vectors. 
\begin{theorem}[Jezierski, \cite{jez90} Lemma 7.3]\label{toriformula}
Let $f, g: T^r \to T^r$ be maps of the $r$ torus with linearizations $F$ and $G$ respectively.
Then 
$N(f,g) = |\det(G - F)|$. If $\det(G - F) \neq 0$ then the linear maps
$F$ and $G$ have $\#(\Phi(F,G)) = N(f,g)$. \hfill ${\Box}$
\end{theorem}

For maps $f, g$ of tori, we will as above often 
identify $f$ with $F$ and $g$ with $G$, and write $N(F,G)$ for $N(f,g)$.
\begin{definition}\label{weak j}
We say that a pair $f$, $g$ is \emph{coincidence weakly Jiang} provided that either
$N(f, g)=0$ or else $N(f, g)= R(f, g)$. If all pairs of maps on a space are coincidence weakly Jiang, we say that
the space itself is coincidence weakly Jiang. 
\end{definition} 

Recall that a Jiang space is one for which the induced map
$x_0^* : \pi_1( X^X, 1_X) \to \pi_1(X, x_0)$ is surjective. These spaces have the property that if the Lefschetz number $L(f,g)=0$, then $N(f,g)=0$, and otherwise $N(f,g)=R(f,g)$. A Jiang space will be coincidence weakly Jiang, and in particular tori are coincidence weakly Jiang.
On the other hand, there are many pairs of maps that will be coincidence 
weakly Jiang, where the spaces are 
not actual Jiang spaces. Our primary example of this  phenomenon occurs 
on the Klein bottle (see Example \ref{KB}). 
\begin{example}\label{counterexamplept1} 
Let $f,g:S^1 \to S^1$ be maps of degree 6 and 2 respectively.
For maps on circles our fundamental groups will be $\Z$, and our maps
$f_*, g_*$ are multiplication by 6 and 2, respectively. By exactness in Theorem \ref{abelianexactseq} we have
$\Reid(f^n_*,g^n_*) \cong \Z_{|6^n-2^n|}$, and since $S^1$ is a Jiang space and 
$L(6^n, 2^n) \neq 0$ we have 
$${\cal R}(f_*,g_*) \cong \Z_4, \ {\cal R}(f_*^2, g_*^2) \cong \Z_{32},
\quad {\cal R}(f_*^3,g_*^3) \cong \Z_{208}, \mbox{ and } {\cal R}(f_*^6, g_*^6)
\cong \Z_{46592}, $$
with respective Nielsen numbers $4$, $32$, $208$ and $46592$. 
\end{example}
\section{Relations among iterates. Geometric and Algebraic reductions. }\label{iteratesection}
We will be giving the coincidence analogues of the foundational 
definitions in periodic point 
theory  in this section. When it is inscapable we will also remind the reader of
the original periodic point concepts.  This is necessary, since we 
want to compare them with the theory we develop here. Our point is 
to indicate
where
straightforward generalizations of periodic point theory to coincidences
fail and why they fail.  

In order to avoid giving too much detail, we are assuming that the reader has 
a basic familiarity with Nielsen periodic point theory.  In this regard, we would point the 
reader to the survey article  \cite{heathfibresurvey}, which we also use as our main
reference.  We will also,  at times,  
need Jiang's original ground breaking work \cite{bojujiang}, as well as early expositions and expansions of it (\cite{hpy, hy}). 
\subsection{Reducible and irreducible Nielsen classes}\label{reduciblesection}
In the preliminary section we outlined existing coincidence Nielsen theory as it applies 
to iterates of self maps
$f, \ g: X \to X$. However, just as Nielsen periodic point theory is much more than the study of the Nielsen numbers of iterates, so our work is more involved than the study of the 
Nielsen coincidence numbers of  iterates. As mentioned in the introduction, our work is
also complicated by a number of obstacles we encounter in our attempt to
generalize periodic point
theory. In Example \ref{mcexample} we saw a case where 
$$ \Phi (f^n, g^n) \neq \bigcup_{m \mid n} \Phi(f^m, g^m).$$ 
On the other hand, the corresponding equality in periodic point theory always holds. 
In fact, the equality 
holds for coincidences when the two maps commute, but we need to state this
in the following
slightly different form:
\begin{lemma}\label{commuting0} If the self maps $f, g$ of $X$ commute, then for all 
positive integers $q$ we have that
\[ \Phi(f^m, g^m) \subset \Phi(f^{qm}, g^{qm}). \]
In particular if $f(x)=g(x)$ for some $x$ then $f^n(x)=g^n(x)$ for all positive $n$.
\end{lemma}
\begin{proof}
The inductive step starts with the proof that $f^m(x) = g^m(x)$ implies that $f^{2m}(x) = g^{2m}(x)$. The rest is straightforward. 
So let $f^m(x) = g^m(x)$, 
then $f^{2m}(x) = f^m(f^m(x))= f^m(g^m(x))= g^m(f^m(x))= g^m(g^m(x))= g^{2m}(x)$.
\end{proof}

This respects Nielsen equivalence to give:
\begin{lemma}\label{commuting1}
If the self maps $f, g$ of $X$ commute, then
for all $m\mid n$, the inclusion map induces a function 
$\gamma_{m,n}: \Phi(f^m, g^m )/\sim \ \to \Phi(f^n, g^n )/\sim$ which takes the class of $x$ at the $m$th level to the class of $x$ at the $n$th level.
\end{lemma}
\medskip

So unlike the fixed point case the $\gamma_{m,n}$ need not exist, and as Example \ref{mcexample} shows it is not enough that the induced homomorphisms
$f_*$ and $g_*$ commute 
(in that example we have $\pi_1(X) \cong \Z$ and so all maps commute in the algebra). 
This will complicate our discussion of the relationship between 
the algebra, which is homotopy invariant, and the geometry, which is not. 
As in the fixed point case the $\gamma_{m,n}$ need not be injective even when they exist
(replace $g$ in Example \ref{mcexample} with the identity, and consider $n=2$). 

There is another immediate obstacle to our attempt to generalize Nielsen 
periodic point theory. In particular Nielsen periodic point theory of a self map $f$
works with $f$ orbits of classes, rather than 
simply with classes. Recall that the period 
of $x \in \Phi(f^n)$, is the 
smallest positive integer $m\mid n$ such that 
$f^m(x) = x$. The list $\{ x, f(x), \cdots, f^{m-1}(x) \}$ is called {\it the orbit} of $x$. The algebraic orbit is the list $\{ \rho_n([x]^n), \rho_n([f(x)]^n), \cdots \}$, and its
length will divide $m$.
Nielsen periodic point theory counts depth of algebraic orbits, rather than classes. This is
because there can be Nielsen equivalences 
(at level $m$) between different
elements $f^i(x)$ and $f^j(x)$ in the above list. When this happens the algebraic length of the orbit (the number of classes in an orbit counted algebraically)
is shorter than the geometric length of {\it points}, and so counting classes gives an inadequate count of the actual minimum number of points present. This idea is
encapsulated in the following fundamental lemma from \cite[Proposition 1.3]{hpy}, and is
the primary reason we consider orbits in Nielsen periodic point theory rather than classes. 
The notation is taken from \cite{heathfibresurvey}, where the angle brackets 
denote $f_*$ orbits of 
classes. We will be discussing algebraic reductions of coincidence classes later in this section. 
\begin{lemma} \label{fundpplemma} (\cite[Proposition 1.3]{hpy})
In Nielsen periodic point theory, 
the (algebraic) length of an orbit $\langle [\alpha]^n \rangle$ divides its depth 
(the minimum integer to which $\langle [\alpha]^n \rangle$ reduces algebraically). If $\langle [\alpha]^n \rangle$
is essential and has depth $d$, then $\langle [\alpha]^n \rangle$ contains at least $d$ periodic points.
\end{lemma}

So then in Nielsen periodic point theory, the use
of orbits is seen to come into its own when the length of orbit is strictly less than its 
depth ($d$ say). At the risk of being repetitious what 
this means is that such orbits contains at least $d$ points, but the number of classes 
is strictly less than $d$. If in the definition of the first periodic point 
number (denoted $NP_n(f)$) we simply counted the number of irreducible essential 
classes, we would be defining a Nielsen number that, in general, had no chance of being
a sharp lower bound. This was the 
fundamental mistake that Halpern made in his famous innovative and useful, but 
unpublished preprint (\cite{halp}). 
The standard example, due to Jiang, comes from a self map of
$\R{\cal P}^3$ (see \cite{bojujiang, hpy}), where for certain $n$ a single irreducible class
(which turns out to be the entire orbit) contains $n$ periodic points. 

One implication of all of this in periodic 
point theory is that the length of an orbit at level 1 must necessarily be 
equal to 1. This can also be seen by the equation $\alpha = \alpha f_*(\alpha)f_*(\alpha^{-1})$,
which shows that $\alpha$ and $f_*(\alpha)$ are Reidemeister equivalent at level 1. 
This is not the case for
iterates of coincidence classes, and the
difficulty in producing a cohesive theory of orbits in a Nielsen theory of 
coincidences of iterates is revealed at this very first level (level 1).
To say more, let $x \in \Phi(f, g)$. Consider %We can certainly form
the list $\{ x, f(x), \cdots, f^{m-1}(x), \cdots \}$, or we could look at the list 
$\{ x, g(x), \cdots, g^{m-1}(x), \cdots \}$. We call these lists the \emph{trajectories} of $x$ under $f$ and $g$ respectively. By Lemma \ref{commuting1} these trajectories are 
the same when $f$ and $g$ commute. The trajectories are perhaps the obvious
candidates for orbits in any generalization of periodic point theory,
but they do not have the desired properties. 
In particular, even at the first level the algebraic trajectory length 
need not be 1. %Note also that in periodic point theory the two trajectories
%do not coincide (unless $x$ is a fixed point). We beg the readers indulgence in (deliberately) confusing 
geometric points and classes with algebraic classes in the following example
(in fact each point is in its own class). 
\begin{example} \label{noorbits2} Let 
$X = S^1$, and define self-maps $f$ and $g$ of $X$ by
$f(e^{i\theta}) = e^{4i\theta}$, and $g(e^{i\theta}) = e^{-3i\theta}$. Then $f$ and $g$ commute. Now $\Phi(f, \ g) = \{e^{2k\pi i/7} \mid k = 0, 1, \cdots, 6 \}$, and the trajectory 
of $2\pi i/7$ is the set
$ \{ e^{2\pi i/7}, e^{8\pi i/7}, e^{4\pi i/7} \}$. 
\end{example}

Since we cannot use the notion of orbit to count coincidence points, we must fall 
back on counting appropriate classes. In light of the fundamental lemma for periodic points
(Lemma \ref{fundpplemma}) this may appear as a severe disadvantage. In fact for the vast 
majority of the spaces we use in our examples this is not the case. This is because
in our examples when the appropriate ordinary coincidence Nielsen numbers
are non-zero the classes can be made into singletons. Actually the only exception
is Example \ref{projsp} which we use to illustrate that we can define orbits when $g$ is invertible. Without going into too much detail, the point is that the orbit definition has 
no advantage over definitions that count classes when the spaces are tori or nil or solvmanifolds. 
The technical name for this in periodic point theory
is essential torality (see Section \ref{npsection}). This definition does not make sense in our theory, but
what we want to 
say is that for these spaces even if we could define orbits, it would not increase the 
number of coincidence points we could detect (see Remark \ref{weckenremark2}). 
In other words for these spaces there would be no advantage in using orbits.

\subsection{Reducible and irreducible Reidemeister  classes}\label{reduciblesection2}
We come now to the algebraic counterpart $\iota_{m,n}$ of the 
geometric ``boosting functions''
$\gamma_{m,n}$. As with the $\gamma_{m,n}$, we need conditions on $f$ and $g$ in order 
for the $\iota_{m,n}$ to be well defined on (Reidemeister) classes.
It is of course an algebraic condition, 
 and what we require here is only that the
induced homomorphisms 
$f_*, \ g_*:\pi_1(X) \to \pi_1(X)$ commute.  
\begin{definition}\label{boostingfunctions}
Suppose for given self maps $f, \ g:X \to X$ that the induced homomorphisms 
$f_*, \ g_*:\pi_1(X) \to \pi_1(X)$ commute. Let $m\mid n$ be integers, 
then the {\it $m$ to $n$ level coincidence boosting functions} (or simply boosting functions)
$\iota_{m,n}$ are defined by the equation
\begin{align*}
\iota_{m,n}(\alpha) &= \Pi_{\ell =0}^{\frac{n}{m}} g_*^{n-(\ell + 1)m}f_*^{\ell}(\alpha) 
\left(= \sum_{\ell =0}^{\frac{n}{m}} g_*^{n-(\ell + 1)m}f_*^{\ell}(\alpha) \mbox{ when $\pi_1$ abelian}\right ) \\
&= g_*^{n-m}(\alpha) g_*^{n-2m}f_*^{m}(\alpha)\cdots 
g_*^{m}f_*^{n-2m}(\alpha)f_*^{n -m}(\alpha).
\end{align*}
\end{definition} 
\begin{lemma} When  $f_*$ and $g_*$ commute, then the $\iota_{m,n}$ are  
 well defined on Reidemeister classes. \hfill ${\Box}$
\end{lemma} 
We abuse notation and use $\iota_{m,n}$ to denote boosting function on both $\pi_1(X)$ and on Reidemeister 
classes. 
\begin{example}\label{counterexamplept3} Continuing Example 
\ref{counterexamplept1} we considered maps of $S^1$ of
degrees 6 and 2 which clearly commute at the level of $\pi_1(S^1)$.
From now on we will identify each of these maps with its 
respective integer. In addition in this example, the 
boosting functions $\iota$ can be represented by multiplication by
an integer (mod $(6^n - 2^n)$), and we will further abuse notation by identifying them
with the said integer. Thus 
\begin{align*}
&\iota_{1,6} = 6^5 + 2\cdot 6^4 + 2^2 \cdot 6^3 + 2^3 \cdot 6^2 + 2^4
\cdot 6 + 2^5 = 11648, \\
&\iota_{2,6} = 6^4 + 2^2 \cdot 6^2 + 2^4 = 1456,\\
\mbox{ and } \ \ \  &\iota_{3,6} = 6^3 + 2^3 = 224. 
\end{align*}
\end{example}

The proof of the following lemma is an easy generalization of the periodic point
case in
\cite[Lemma 3.1]{hhk2}.
\begin{lemma} {\bf On the nose boosting.} \label{noseboost} 
Let $[\alpha]^n\in \Reid(f^n,g^n)$ be a Reidemeister class 
that reduces to a class 
$[\beta]^m$ at level $m$. Then for any $\sigma \in [\alpha]^n$, there is a $\tau \in [\beta]^m$ for which
$\iota_{m,n}(\tau) =\sigma$ in $\pi_1(X)$. 
\hfill ${\Box}$
\end{lemma}

The following can be verified by an easy calculation from the definitions of the 
$\rho_k$, the geometric and algebraic boosts and Lemma \ref{commuting1} (see \cite[Proposition 1.14]{hpy} for the corresponding 
proof in periodic point theory).
\begin{lemma}\label{diagrammatic}  %WARNING notation
Suppose that $f, \ g:X \to X$ induce commuting homomorphisms on $\pi_1(X)$. Let $k\mid m \mid n$ be integers, then we have that $\iota_{k,n}=\iota_{m,n}\iota_{k,m}$.%, and 
%that $\iota_{m,n}(g_*^m\cdot^{-1} f_*^{m})= g_*^{n} \cdot^{-1} f_*^{n}$. 
Furthermore if $f$ and $g$ commute as functions, then 
the following diagram exists 
\[ \begin{CD}
\Phi(f^m,g^m)/\sim @>\rho_m>> \Reid(f_*^m, g_*^m) \\
@V\gamma_{m,n}VV @VV\iota_{m,n}V \\
\Phi(f^n,g^n)/\sim @>\rho_n>> \Reid(f_*^n, g_*^n) ,  
\end{CD}
\]
and is commutative. \hfill ${\Box}$
\end{lemma}
\begin{definition} 
\label{ER}
Let $f, \ g: X\to X$ be maps. We say 
$[\alpha]^n \in {\cal R}(f_*^n , g_*^{n})$ is \emph{reducible}
to $[\beta]^m \in {\cal R}(f_*^m, g_*^m )$ if
$\iota_{m,n}([\beta]^m) = [\alpha]^n$. If $[\alpha]^n$ is not reducible to any level $m<n$ then it is \emph{irreducible}. 
We say that 
$[\alpha]^n$ has \emph{depth} $d$ if $d$ is the smallest integer for which
there is a class 
$[\delta]^d$ to which $[\alpha]^n$ reduces. 
\end{definition}

As in periodic point theory, if $[\alpha]^n $ is in the image of no 
$\iota_{m,n}$ for any $m\mid n$, then there can be no geometric coincidence points of $f^m$ and $g^m$ whose Reidemeister class in $\Reid(f_*^m, g_*^m)$ boosts to $[\alpha]^n$. In
fact there can be no coincidence points of $f_1^m$ and $g_1^m$ whose class boosts to $[\alpha]^n$ for any $f_1\simeq f$ and $g_1\simeq g$.

When $f$ and $g$ commute, this fact follows from the diagram above, but note that we do not need commutativity of $f$ and $g$ (only of $f_*$ and $g_*$) to deduce this. So then our
algebraic constructions are still useful in detecting geometric behaviour even when the maps are not geometrically commutative.
\section{The analogue $NP_{n}(f,g)$ of the periodic point number $NP_{n}(f)$.}\label{npsection}
As in the introduction 
we use the symbols $P_{n}(f,g)$ to denote the set of points $x$ with 
$f^n(x) = g^n(x)$ but $f^m(x) \neq g^m(x)$ for any $m\mid n$, and 
$MP_{n}(f,g)$ to denote the minimum $\min \# \{P_{n}(f_1,g_1) \mid f_1 \sim f \mbox{ and } 
g_1 \sim g\}$. 
The aim in this section is to define a suitable lower bound for $MP_{n}(f,g)$, give some of its 
properties together with a number of examples. 
\begin{definition}\label{NP}
We define $NP_{n}(f,g)$ as the number of irreducible essential Reidemeister classes of $\Reid(f_*^n, g_*^n)$. 
\end{definition}
\begin{theorem}\label{northbound}
$NP_{n}(f,g)$ is homotopy invariant in $f$ and $g$, and 
\[ NP_{n}(f,g) \leq MP_{n}(f,g) \leq \#(P_{n}(f,g)). \]
\end{theorem}
\begin{proof}
The homotopy invariance holds because $NP_{n}$ is defined only in terms of the induced maps $f_*$ and $g_*$, and the inequality on the right is obvious.

For the other inequality, let $f$ and $g$ be arbitrary. We need only show that 
$NP_{n}(f,g) \leq \#(P_{n}(f,g)).$ Accordingly, let 
$[\alpha]^n \in \Reid(f_*^n, g_*^n)$
be an essential irreducible class. Because $[\alpha]^n$ is essential, there is a coincidence 
point $x$ with coincidence class 
${\bf A}^n$ with $f^n(x) = g^n(x)$ and $\rho_n({\bf A}^n) =[g^n(c)f^n(c^{-1})]^n = [\alpha]^n$,
where $c$ is any path from $x_0$ (the base point) to $x$. It suffices to show that
$x \in P_{n}(f,g)$. If this is so, then
this process will define an injection from the irreducible essential
Reidemeister classes into $P_{n}(f,g)$ establishing the inequality.

For the sake of deducing a contradiction, assume that $x \not \in P_{n}(f,g)$, that is, there is some $m$ such that $f^m(x) = g^m(x)$ with $m\mid n$, and $m \neq n$. 
Then we will have $x \in {\bf B}^m$ where ${\bf B}^m$ is the coincidence class of $x$ at level $m$. Let $[\beta]^m = \rho_m({\bf B}^m) =[g^m(c)f^m(c^{-1})]^m$. Then
by the definition of $\iota_{m,n}$ we have that
$\iota_{m,n}(\rho_m({\bf B}^m))= \iota_{m,n}([g^m(c)f^m(c^{-1})]^m)=[g^n(c)f^n(c^{-1})]^n = [\alpha]^n$, contradicting the irreducibility of $[\alpha]$.
\end{proof}

Our first task is to compare $NP_n(f, \id)$ with $NP_n(f)$ ($\id$ is the identity). 
Recall, in the context of periodic point theory,
that a map $f:X \to X$ is said to be {\it essentially toral} (\cite{hk1})
if, for all $m \mid n$ and 
every $[\alpha]^m\in {\cal E}(f^m)$, the depth of $[\alpha]^m$ and the orbit length 
of $[\alpha]^m$ coincide, and 
the boosting functions are injective on essential boosts. As already discussed, the first 
part of this simply  means we may as well define $NP_n(f)$ to be the number of 
irreducible essential {\it classes}.  For all but one of our 
examples (Example  \ref{projsp}), the maps involved are  (individually) essentially toral. 
see \cite[Corollary 4.6]{hk1}

\begin{proposition}\label{pfg=pf} Let $f:X \to X$ be a self map, then
$NP_n(f) \geq NP_n(f,\id)$.
 If $f$ is essentially toral, then $NP_n(f,\id) = NP_n(f)$. \hfill ${\Box}$
\end{proposition}

At this point in the exposition of periodic point theory in
\cite{hpy} we would be working towards a M\"obius inversion type
result on a certain class of maps on Jiang spaces, in particular on tori.
The generalization of this result applied to Example \ref{counterexamplept1}, would say that 
$NP_6( f, \ g)= N(f^6, \ g^6) - N(f^3, \ g^3) - N(f^2, \ g^2) + N(f, \ g)$. As we shall see below, the coincidence version 
of this is false in general 
(even when $f$ and $g$ commute). 
We do have a weaker version of this result,  but before we can state it, we need the following  coincidence analogue of a defintion  from \cite{hk1}. 
\begin{definition}(c.f. \cite{hk1})
We say that the pair $f$, $g$ is {\it
coincidence essentially reducible} provided that for any 
essential class
$[\alpha]^n$ of $f^n$ and $g^n$, if
$[\alpha]^n$ reduces to some class $[\beta]^m$, then
$[\beta]^m$ is also essential. 
If for a given space $X$ any pair of self maps is
coincidence essentially reducible then we say that $X$ is coincidence essentially reducible. 
\end{definition}

For $g= \id$ the identity on a torus, the pair $f$, $g$ 
 is always coincidence essentially reducible. 
In periodic point theory there are simple 
examples of maps that are not essentially reducible, but these tend to be maps on
non manifolds, and so do not provide examples in our setting (the coincidence index is 
in general not defined for non-manifolds).  For coincidences of iterates, we need 
the assumption that the induced maps commute at the level of the fundamental group. 
\begin{theorem}\label{thmessred} If $f$ and $g$ are maps of
tori for which the induced maps on the
fundamental groups commute, then the pair $f$, $g$ 
is coincidence essentially reducible.
\end{theorem}

The proof in \cite{hk1}, 
that tori are essentially reducible for periodic point theory, 
uses the linearization of the self map
$f$ under consideration. In particular it uses the linearization $F$ of $f$, and the fact that 
$N(f^n)= |\det(F - I)|$, where $I$ is the identity matrix. We need a slightly different proof than 
the one given in \cite{hk1},
but we will use Theorem \ref{toriformula}. We also use the fact that with respect to both
fixed and coincidence point 
theory tori are Jiang spaces and hence also weakly Jiang (Definition \ref{weak j}). 

\begin{proof}
Suppose that some class at level $m$ say, is essential. Then 
every class at level $m$ is essential. We show that every class at level $k\mid m$ is also essential, 
or equivalently that $N(f^k, \ g^k) \neq 0$ . 
Accordingly let $F$ and $G$ be the linearizations of $f$ and $g$ respectively. 
Now we have that $N(f^m, \ g^m)= |\det(F^m - G^m)|\neq 0$. Let $r = \frac{m}{k}$. 
By hypothesis $F$ and $G$ commute, 
so $F^m - G^m= (F^k - G^k) (F^r + F^{r-1}G + \cdots + G^r)$. So 
$N(f^k, \ g^k)= |\det(F^k - G^k)| \neq 0$ or else we would have that 
$|\det(F^m - G^m)|=0$, a contradiction.
\end{proof} 

The next result is our weaker version of 
M\"obius inversion, which requires only essential reducibility. The periodic point analogy of the second inequality holds true, but this is not the case for the first inequality (see also Example \ref{projsp}, and the discussion in Section \ref{invertiblesection}).
\begin{theorem}\label{pseudomob}
Let $P(n) = \{ p(1), \ p(2), \cdots, p(k) \}$ be the set of primes dividing $n$, and suppose that 
the pair $f$, $g$ is essentially reducible. Then 
$$N(f^n, \ g^n) \geq NP_n(f,g) \geq N(f^n, g^n) - \sum_{i = 1}^k N(f^{n:i}, g^{n:i}), $$
where ${n:i}= n \cdot (p(i))^{-1}$. 
\end{theorem}
\begin{proof}
If $N(f^n, g^n) = 0$ there is nothing to prove. When $N(f^n, g^n) \neq 0$
we write ${\cal E}(f^n, g^n)= {\cal IE}(f^n, g^n)\cup {\cal RE}(f^n, g^n)$, where 
${\cal IE}(f^n, g^n)$ is the set (with cardinality $NP_n(f, g)$) of essential 
irreducible algebraic classes in ${\cal R}(f_*^n, g_*^n)$. The set $ {\cal RE}(f^n, g^n)$ consists
of the reducible essential classes, and of course the union is disjoint. 

We show that $\# {\cal RE}(f_*^n, g_*^n) \leq 
\#\bigsqcup_{i= 1}^k {\cal E}(f^{n:i}, g^{n:i}) $ (where $\sqcup$ denotes disjoint union). 
In fact we construct an injection
$\psi: {\cal RE}(f_*^n, g_*^n) \to \bigsqcup_{i= 1}^k {\cal E}(f^{n:i}, g^{n:i})$. 
So let $[\alpha]^n \in {\cal RE}(f_*^n, g_*^n)$, and let $m\mid n$ with $m \neq n$ 
be the maximal 
integer for which there exists a 
$[\beta]^m \in {\cal E}(f^{m}, g^{m})$ with $\iota_{m,n}([\beta]^m)= [\alpha]^n$. 
Necessarily $m = n:i$ for some $i$. 
Define $\psi([\alpha]^n) = [\beta]^m$. Clearly $ \iota_{m, n}\psi$ is the identity, 
thus $\psi$ is injective, and 
$\# {\cal RE}(f_*^n, g_*^n) \leq \sum_{i= 1}^k \# {\cal E}(f^{n:i}, g^{n:i}) $. 
Thus 
\begin{align*}
N(f^n, g^n) & = NP_n(f,g) + \#{\cal RE}(f^n, g^n) \\
&\leq NP_n(f,g) + \sum_{i= 1}^k \# {\cal E}(f^{n:i}, g^{n:i}) \\
& = NP_n(f,g) + \sum_{i= 1}^k N(f^{n:i}, g^{n:i}) ,
\end{align*} which implies the result. 
\end{proof}

Actually Theorem \ref{pseudomob} gives a new result in periodic point theory namely:
\begin{corollary} If $f: X \to X$ is a map of a solvmanfiold $X$, then with the notation of
Theorem \ref{pseudomob} we have that 
$$N(f^n) \geq NP_n(f) \geq N(f^n) - \sum_{i = 1}^k N(f^{n:i}). $$
\end{corollary}

The following example further illustrates Theorem \ref{pseudomob}, and will provide 
a counterexample to M\"obius inversion of the type found in \cite{hpy}.
\begin{example}\label{counterexamplept4} Continuing examples
\ref{counterexamplept1} and \ref{counterexamplept3}
where we considered maps of $S^1$ of
degrees 6 and 2 with $n= 6$ then from Theorem \ref{pseudomob} we have that 
$$46592 \geq NP_6(f,g) \geq 46592 - 32 - 208 = 46352.$$
We now show that $ NP_6(f,g) \neq N(f^6, \ g^6) - N(f^3, \ g^3) - N(f^2, \ g^2) + N(f, \ g)= 
46352 + 4 = 46356$.
From the proof of theorem \ref{pseudomob}, we have that 
$NP_6(f,\ g)= R(f^n, g^n) - \#\bigcup_{m\mid n} \im(\iota_{m,n})
= 46592- \# (\im (\iota_{2, 6}) \cup \im(\iota_{3,6}))$. ($\im$ denotes the image of a homomorphism.) So in particular, 
we need to compute the 
cardinality of the intersection $\im(\iota_{2, 6}) \cap \im(\iota_{3,6})$. 
In \ref{counterexamplept3} we computed the boosts from 2 to 6 and from 3 to 
6 to be multiplication by 1456 and 224 respectively.
Thus 
$\im(\iota_{3,6})$ is the subgroup of 
$\Z_{46592}$ of order 208 generated by 224, and similarly 
$\im(\iota_{2,6})$ is the subgroup of 
$\Z_{46592}$ of order 32 generated by 1456. So 
$\im(\iota_{2, 6}) \cap \im(\iota_{3,6})$ is the subgroup of 
$\Z_{46592}$ generated by $2912$, the least common
multiple of $224$ and $1456$. This subgroup is of order $46592/2912= 16$. So then by the principle of inclusion and exclusion $NP_6(f,\ g)= 46592 - 32 - 208 + 16= 46368 \neq 46356$. 
\end{example}

We now start to work our way towards a full M\"obious inversion formula. In many ways
we are generalizing directly from \cite{hk1}, but of course looking at classes not orbits. 
This changes the definitions slightly. 
\begin{definition} 
Let $f, g:X \rightarrow X$ be coincidence essentially reducible. We 
say that the pair $f$, $g$ is {\it injective on essential boosts to level $n$} if for all $m\mid n$ and 
for any classes 
$[\beta_1]^m$, $[\beta_2]^m$ at level $m$ with $\iota_{m,n} ([\beta_1]^m) =
\iota_{m,n}([\beta_2]^m) \in {\cal E}(f^n, g^n)$ then we have that 
$[\beta_1]^m =[\beta_2]^m$. If $f$, $g$ are injective on essential boosts for all $n$, we say that is {\it $f$, $g$ are injective on essential boosts}. 
\end{definition} 
\begin{lemma}
If $\pi_1(X)$ is abelian, $f_*$ and $g_*$ commute, and $\Coin( f^n_*, g^n_*)= 0$, then 
the pair $f$, $g$ is injective on essential boosts to level $n$. In particular if $X$ is a torus
and $\det(F^n-G^n) \neq 0$, then the pair $F$, $G$ is injective on essential boosts to level $n$.
\end{lemma}
\begin{proof} We assume that $N(f^n, g^n) \neq 0$ or there is nothing to prove.
In this case $ {\cal E}(f^m,g^m)= {\cal R}(f_*^m,g_*^m)$.
Since $\pi_1(X)$ is Abelian, for all $m|n$ the functions
$\iota_{m,n}$ are homomorphisms, and 
from Theorem \ref{abelianexactseq} since $\Coin( f^n_*, g^n_*)= 0$, we have that 
$g^n_* - f^n_*: \pi_1(X) \to \pi_1(X)$ is injective. 
Now let $m\mid n$ and $[\beta]^m \in {\cal E}(f^m,g^m)$ be such that 
$\iota_{m,n}([\beta]^m)= [0]^n$. Since the $\iota_{m,n}$ are homomorphisms we need only 
show that $[\beta]^m= [0]^m$.

So let $\beta \in [\beta]^m$. Since $\iota_{m,n}([\beta]^m)= [0]^n$, then
$\iota_{m,n}([\beta]^m) \in Ker(j_n)= Im(g^n_* - f^n_*)$ from exactness in Theorem
\ref{abelianexactseq}. So then there is a $\gamma \in \pi_1(X)$ such that 
$(g^n_* - f^n_*)(\gamma) = \iota_{m,n}(\beta)$. Composing with $g_*^n - f_*^n$ we have: 
\[
(g_*^n - f_*^n)(g_*^m - f_*^m)(\gamma) =
(g_*^m - f_*^m)(g_*^n - f_*^n)(\gamma) = 
(g_*^m - f_*^m)\iota_{m,n}(\beta)=(g_*^n - f_*^n)(\beta). 
\]
But $g^n_* - f^n_*$ is injective, 
so actually $\beta= (g^m - f^m)(\gamma)$. But this means (again from exactness in Theorem 
\ref{abelianexactseq}) that $[\beta]^m = [0]^m$ as required. 

For tori, when $\det(F^n-G^n) \neq 0$, then all classes at level $n$ are essential. Since 
$\pi_1(X) \cong \Z^r$ for some $r$ we have
$Ker(g^n_* - f^n_*) \cong Coin(f_*^n, g_*^n)= 0$, and the result
follows from the first part.
\end{proof} 
\begin{definition} \label{coincidence essentially
reducible}
Let $f, \ g:X \rightarrow X$ be maps. We 
say that the pair $f$, $g$ is {\it coincidence essentially
reducible to the gcd,} if they are
coincidence essentially reducible, 
and 
whenever 
$[\alpha]^n\in{\cal E}(f^n, g^n )$ 
reduces to both $[\beta]^m\in{\cal E}(f^m, g^m )$
and $[\gamma]^k\in {\cal E}(f^k, g^k )$, then 
there is a $[\delta]^d\in {\cal E}(f^d, g^d )$ 
with $d = \gcd(m,k)$ to which both 
$[\beta]^m$
and $[\gamma]^k$ reduce. 
If every pair $f,g$ is coincidence essentially reducible to the gcd, 
we say that $X$ is \emph{coincidence essentially 
reducible to the gcd}. 
\end{definition} 
\begin{example}\label{counterexamplept5} The continuing examples
\ref{counterexamplept1}, \ref{counterexamplept3} and \ref{counterexamplept4}, where
$f= 6$ and $g= 2$ show
that not all maps of tori satisfy the above definition. It is here that the M\"obius 
formula breaks down. In particular if this example were essentially reducible to the gcd 
then we would have that the intersection $\im (\iota_{2, 6}) \cap \im(\iota_{3,6})$, 
would coincide with $\im (\iota_{1,6})$ which as \ref{counterexamplept4} shows it 
does not. 
\end{example}

We will prove the following theorem in the appendix.
\begin{theorem} \label{circlegcd}
Let $f, g:S^1 \to S^1$ be maps of degrees $a,b \in \Z$ respectively. Then
$f,g$ is coincidence essentially reducible 
to the GCD if and only if $\gcd(a,b)=1$. 
\end{theorem}
\begin{example}\label{KB} {\bf The Klein Bottle example Part I.}
Let $K^2$ denote the Klein Bottle. 
We regard $K^2$ as the quotient space of ${\R}^2$
under the equivalence relation defined by
$(s,t)\sim((-1)^ks,t+k)$ and $(s,t)\sim(s+k, t)$ for any $k\in{\Z}$. The Klein bottle
fibres as $S^1 \to K^2 \stackrel{p}{\to} S^1$, where $p$ is induced by projection on the 
second factor.
Given any pair of 
integers $q, \ r$ for which $r$ is odd, 
or $r$ is even and $q = 0$, 
the correspondence $(s,t)\to(qs,rt)$ modulo the equivalence relation defined above,
induces a 
well defined, fibre preserving map on $K^2$. We abuse notation
and denote this map by $(q,r)$.
Since $p$ is the projection on the second factor, the map on the base 
is the standard map of degree $r$, and 
the restriction to the principle fibre over the base point has degree $q$. 
We point the reader to 
\cite{hk1, hk2} for details of all this. 
Let 
$(a, c)$ and $(b, d)$ be two such well defined maps on $K^2$. If $\gcd(a,b )= 1$ and 
$\gcd(c,d) = 1$, then the pair of maps $(a, c)$ and $(b, d)$ 
is injective on essential boosts and
essentially reducible to the gcd. 
\end{example} 
\begin{remark} A rigorous verification of this example is beyond the scope of this
paper. The proofs however do not contain anything that is really new. They rely on coincidence versions of the
fibre techniques used in \cite{hk1}, where a number of different 
properties of
the fibre preserving maps are deduced from the 
very same properties on the fibres and the base
(see for example \cite[4.4, 4.12]{hk1}). 
\end{remark}

In Section \ref{invertiblesection}
we will indicate the general theme in the coincidence theory of a pair $f,g$ that 
strong results of periodic point theory hold when $g$ is invertible. 
The following theorem shows that simply requiring $g_*$ (but not $g$) to be invertible 
can also often yield strong results.
\begin{theorem} \label{gcdmatrix} Suppose that $\pi_1(X)$ is Abelian, the pair $f$, $g$ 
is coincidence essentially reducible, that 
the pair $f_*$, $g_*$ commute at the level of $\pi_1(X)$, are injective on essential boosts,
and that $g_*: \pi_1(X) \to \pi_1(X)$ is invertible. 
Then the pair $f$, $g$ is coincidence
essentially reducible to the gcd.
\end{theorem}

The proof below is a modification of Boju Jiang's proof  of a similar result in
periodic point theory (see \cite[Proposition 4.4]{bojujiang}). 
  Our analogue of reducible to the gcd is more specific than Jiang's
concept, in that he does not require $\delta$ in the proof below to
boost to $\beta$ and $\gamma$.
It is here where we need injectivity on essential boosts. 
\begin{proof}
For this proof (and actually for the Appendix as well), we want to 
refer both to coincidence boostings and periodic point boostings. 
So for these two places, we use the notation 
$\iota^{f_*,g_*}_{p,q}$ to refer to the sum (product) in Definition \ref{boostingfunctions}
($\pi_1$ is Abelian here), and we use the 
notation $\iota^{t,1}_{p,q}$ to refer to the sum (to the corresponding 
product in Section \ref{invertiblesection}) 
$$\iota^{x,1}_{p,q} = 1 + x^p + x^{2p} \cdots + x^{q-p},$$
where $x$ is an arbitrary function on $\pi_1(X)$ and of course $1$ is the identity function.
Now let 
$[\alpha]^n \in {\cal R}(f_*^n, \ g_*^n)$ be essential and 
reducible to both $[\beta]^{k} \in {\cal R}(f_*^{k}, \ g_*^{k})$,
and $[\gamma]^m\in {\cal R}(f_*^m, \ g_*^m)$. As in \cite{bojujiang} we may without loss
assume that $d= \gcd(k, m) =1$ (or we may can work with the maps $f^d$ and $g^d$). 
By Lemma \ref{noseboost}, we may assume without loss, that
we have representatives $\alpha$, $\beta$ and $\gamma$ of
$[\alpha]^n$, $[\beta]^{k}$ and $[\gamma]^m$ respectively, in $\pi_1(X)$
such that 
$\iota^{f_*,g_*}_{k, n}(\beta) = \alpha$ and $\iota^{f_*,g_*}_{m, n}(\gamma) = \alpha$.
Now since $g_*$ is invertible, we have that
\begin{align*}
\iota^{f_*,g_*}_{k, n}(\beta) &= \left(
g_*^{n-k} + g_*^{n-2k}f_*^{k} + \cdots + 
g_*^{k}f_*^{n-2k}+ f_*^{n -k}\right)(\beta) \\
&= g_*^{n-k}\left(
1 + (g_*^{-1}f_*)^{k} + \cdots + 
(g_*^{-1}f_*)^{n-2k}+ (g_*^{-1}f_*)^{n -k}\right)(\beta). \\
&= g_*^{n-k}\iota^{g_*^{-1}f_*,1}_{k,n}(\beta)= \alpha.
\end{align*}

Similarly $\iota^{f_*,g_*}_{m, n}(\gamma)=
g_*^{n-m}\iota^{g_*^{-1}f_*,1}_{m,n}(\gamma) = \alpha$, 
and $\iota^{f_*,g_*}_{1, n}(\delta)= g_*^{n-1}\iota^{g_*^{-1}f_*,1}_{1,n}(\delta ) $
for any $\delta$. 

Since $\gcd(k, m) = 1$ we may again without loss, assume we are given 
positive integers $a$ and $b$ with $ak - bm= 1$. 
Let $P(x) = 1 + x^{k}+ ... + x^{(a-1)k}$ and $Q(t) = -x(1 + x^m + ... + x^{(b-1)m})$, then 
$$P(x) (1 + x + \cdots + x^{k -1}) + Q(x) (1 + x + \cdots + x^{m -1}) =1, $$
or $P(x)\iota^{x,1}_{1,k}(\gamma) + Q(x)\iota^{x,1}_{1,m}(\gamma)= \gamma$ for all $x$.

Set
$\delta = P(g_*^{-1}f_*)g^{- k+1} (\beta) + Q(g_*^{-1}f_*)g^{-m +1}(\gamma)$. We will
use $x= g_*^{-1}f_*$ below. Note that this commutes with everything in sight, and we have
\begin{align*}
\iota^{f_*,g_*}_{1, n}(\delta) &= g_*^{n-1} \iota^{g_*^{-1}f_*,1}_{1,n}(\delta) \\
&= g_*^{n-1} \iota^{g_*^{-1}f_*,1}_{1, n} P(g_*^{-1}f_*)g^{- k +1} (\beta) + 
g_*^{n-1} \iota^{g_*^{-1}f_*,1}_{1, n}Q(g_*^{-1}f_*)g^{-m +1}(\gamma) \\
&= P(g_*^{-1}f_*) g_*^{n-k} \iota^{g_*^{-1}f_*,1}_{1,k} 
\iota^{g_*^{-1}f_*,1}_{ k,n} (\beta) + 
Q(g_*^{-1}f_*) g_*^{n-m}\iota^{g_*^{-1}f_*,1}_{1, m}\iota^{g_*^{-1}f_*,1}_{m, n}(\gamma) \\
&= P(g_*^{-1}f_*)
\iota^{g_*^{-1}f_*,1}_{1,k} 
g_*^{n- k} \iota^{g_*^{-1}f_*,1}_{ k,n} (\beta) +
Q(g_*^{-1}f_*)\iota^{g_*^{-1}f_*,1}_{1, m})
g^{n- m}\iota^{g_*^{-1}f_*,1}_{m, n}(\gamma) \\
&= P(g_*^{-1}f_*)
\iota^{g_*^{-1}f_*,1}_{1,k} (\alpha) +
Q(g_*^{-1}f_*)\iota^{g_*^{-1}f_*,1}_{1, m}(\alpha)= \alpha.
\end{align*}
Since the pair $f$, $g$ is coincidence essentially reducible, then $\delta$ is 
essential. 
That $\iota^{f_*,g_*}_{1, k}(\delta) = \beta$ and $\iota^{f_*,g_*}_{1, m}(\delta) = \gamma$ 
follows from injectivity on essential boosts.
\end{proof}

The 
periodic point analogue of the 
following theorem can be found in \cite{hk1}. Since the proof 
contains no new ideas, it is omitted. 
\begin{theorem} 
\label{npeqn} Suppose that the pair $f$, $g$ is coincidence essentially reducible to the gcd, 
that the
pair $f_*$, $g_*$ commute, and are injective on essential boosts. If $f^n, \ g^n$ are coincidence weakly Jiang
and $N(f^n, g^n)\neq 0$, 
then for all $m\mid n$ we have
%$$N\Phi_m(f)=N(f^m).$$ 
%Furthermore, in this case 
%M\"obius inversion applies to give that for all $m|n$ 
$$NP_m(f)=\sum_{\tau\subset{\bf p}(m)}(-1)^{\#\tau}N(f^{m:\tau}, \ g^{m:\tau})$$ 
where ${\bf p}(m)$ denotes the set of prime divisors of $m$ and 
$m:\tau=m\prod_{p\in\tau}p^{-1}$. \hfill ${\Box}$
\end{theorem} 
\begin{example}\label{KB2} 
{\bf The Klein Bottle example Part II}
Continuing example \ref{KB} assume that we have maps 
$(a, c)$ and $(b, d)$ with $\gcd(a,b )= 1$ and 
$\gcd(c,d) = 1$, then from Example \ref{KB} the pair is injective on essential boosts, 
and essentially reducible to the gcd. Moreover 
$a \neq \pm b$ and $c\neq \pm d$, and a 
coincidence version of \cite[corollary 4.26]{hk1} will give that 
the pair 
$(a, c)$, $(b, d)$ is weakly Jiang with $N((a, c)^n, \ (b, d)^n) \neq 0$ for any $n$. Thus
the pair $(a, c)$, $(b, d)$
satisfies the hypothesis of \ref{npeqn} for all $n$. 
Using the na\"ive addition formula for coincidences (\cite{jez90, heathfibresurvey} and
coincidence distribution
considerations analogous to 
\cite[theorem 4.8]{hk3} we have
$$N((a, c)^n, \ (b, d)^n) =\frac{|c^n-d^n|}{2}\sum_{j=0}^1|a^n +(-1)^jb^n)|. $$
Let $(a,b )= (2, 3)$ and 
$(c,d) = (3, 5)$, then we have 
\begin{align*}
N(f,\ g)&= \frac{2}{2}(1 + 5)=6, 
&N(f^2,\ g^2)&= \frac{16}{2}(5+ 13)= 144, \\ 
N(f^3,\ g^3)&= \frac{98}{2}(19 + 35)= 266, & N(f^6,\ g^6)&= \frac{14896}{2}(665 + 793)= 10859184, 
\end{align*}
so $NP_6(f, \ g) = 10856400$.
\end{example} 
\begin{example}\label{matrixexample1} Consider the following commuting matrices
$$ F= \begin{bmatrix}-2& 2 \\ 1& 2 \end{bmatrix} \quad \mbox{and} \quad
G= \begin{bmatrix}-1& 0 \\ 1& 1 \end{bmatrix}, $$
which we regard as maps of $T^2$. Note that $G$ is invertible over $\Z$,  and 
that $F$ and $G$ commute. So 
 $F$ and $G$ are essentially 
reducible to the gcd at level $n$ by 
Theorem \ref{gcdmatrix}.
Recall that $N(F^n, G^n) = |\det(F^n - G^n)|$,
so by Theorem \ref{npeqn} we have
\begin{align*}
NP_{30}(F,G) & = |\det(F^{30} - G^{30})| - |\det(F^{15} - G^{15})| - |\det(F^{10} - G^{10})|
- |\det(F^{6} - G^{6})| \\
& \quad + |\det(F^{5} - G^{5})| + |\det(F^{3} - G^{3})| + |\det(F^{2} - G^{2})| - |\det(F - G)| \\
& = 221073919719792987930625-470183304961-60450625-46225\\
& \quad +7561 +181 +25 -1 = 221073919719322744136580.
\end{align*}
\end{example}

\section{The analogue $N\Phi_{n}(f,g)$ of the periodic point number $N\Phi_{n}(f)$.}\label{nphisection}
In Nielsen periodic point theory the second number $N\Phi_n(f)$ satisfies:  
\begin{align*} 
\sum_{m |n}NP_m(f) \le N\Phi_n(f) \le M\Phi_n(f) &= \min \{\#\Phi(f_1^{n})\mid f_1 \simeq f \}.\end{align*} 
In the introduction we saw in Example \ref{mcexample}, that the  number 
$MC_n(f,g) = \min \{ \# \Phi(f_1^n, g_1^n) \mid f_1 \simeq f, g_1 \simeq g \}$ 
may not take into account coincidences of
iterates at levels other than
$n$.  So what then is our  second number
$N\Phi_n(f,g)$ to measure, or to put it another way,  of what is $N\Phi_n(f,g)$ to be a lower bound?  In answering this question, we will need to take into account the phenonomon
encountered in the next example,  which shows that we can have 
$$ P_m(f, g)  \cap  P_q(f, g) \neq \emptyset \mbox{ for $m \neq q. $}$$

\begin{example}\label{nondivide}
For this example we consider the circle $S^1$ as the interval $[0,1]$ with endpoints identified. Let $\bar 0= \{ 0,1 \}$ be the base point, and 
$f$ the standard map of degree 2 on the usual presentation of $S^1$ as a quotient of $\R$, conjugated with 
a homeomorphism between the two presentations. 
 Let $\epsilon \in (0,1)$ be the
real number defined below, and let $g$ be the degree one map $g(x) = x^\epsilon$. Note that $g([0,1]) = [0,1]$ with $g(0)=0$ and $g(1)=1$ and so  $g$ is well-defined on $S^1$. 

We will show that there is a point $q \in P_3(f,g) \cap P_5(f,g)$.  %We will 
%also show that as classes 
%$[q]^3 \not \subset [q]^5$. We will draw some conclusions. 
As we will see, there are real numbers $q\approx 0.0349$ and
$\epsilon \approx 0.1265$ satisfying:
\begin{align*}
8q &= q^{3\epsilon}, \\
32q &= q^{5\epsilon} + 1.
\end{align*}
To construct $q$ and $\epsilon$, 
let $z = q^\epsilon$, then $z^5 - 4z^3 + 1= 0$. 
Let $z$ be the real root near $.6541$, then $z^3 = 8q$, so $q = z^3 / 8$. We can then  
 solve for $\epsilon$, and we have that $f^3(q) = 8q = q^{3\epsilon} = g^3(q) \mbox{ and } 
   f^5(q) = 32q = q^{5\epsilon} \mbox{  in   $S^1$, but $f(q) \neq g(q)$. }$
We have shown for this example that  $ P_3(f, g)  \cap  P_5(f, g)\neq \emptyset$.  
\end{example}

The phenonomon of the example cannot of course happen
in periodic point theory.  If $f^3(x) = x = f^5(x)$, then
$f(x)= f^6(x) = f^3(f^3(x))= x$!
We will say more later.  For now  it should be
clear that in defining $N\Phi_{15}(f,g)$ in this example, we 
would need to take account of the classes of $q$ at both level $3$ and at level $5$.  It is this that lies behind our defintion of
$M\Phi_n(f,g)$ which now define formally  as: 
\begin{align*}
M\Phi_n(f,g) &= \min \{ \# \bigsqcup_{m|n} P_m(f_1, g_1) \mid f_1 \simeq f, g_1 \simeq g\}. 
\end{align*}
The main goals of this section are to define a Nielsen type
number $N\Phi_n(f,g) $, show it is a lower bound for $M\Phi_n(f,g)$, 
and give some of its properties together with a  number of examples.
 We adapt the definition of the periodic point number
$N\Phi_n(f)$ taking account of the following: In periodic point theory 
 $N\Phi_n(f)$ is defined in terms of
sets of $n$-representatives of orbits and of the heights and depths therof.
Since (as in the
definition of $NP_n(f,g)$) we cannot use orbits (and so neither height nor 
depth), we simply work with classes. 
Apart from this the definitions are entirely analogous.

We use the symbol $\mathcal E(f^m,g^m)$ to denote the set of all essential coincidence
classes of $f^m$ and $g^m$. Thus $N(f^n, g^n) = \#(\mathcal E(f^m,g^m))$.
For a fixed positive integer $n$, a set 
\[ \mathcal G \subset \bigsqcup_{m \mid n} \mathcal E(f^m,g^m) \]
is called a \emph{ coincidence set of $n$-representatives for $f$ and $g$} if 
each class of
$\bigsqcup_{m \mid n} \mathcal E(f^m,g^m)$ reduces to, 
or is equal to, some element of $\mathcal G$.
\begin{definition}\label{nphin}{\em 
The {\it full Nielsen type number $N\Phi_n(f,g)$} is defined to be 
the minimal size among all sets
of coincidence $n$ representatives. 
}
\end{definition}

\begin{theorem}\label{nphilowerbound} The number
$N\Phi_{n}(f,g)$ is homotopy invariant, and satisfies the inequalities: 
\[  \sum_{m|n} NP_m(f,g))\leq N\Phi_{n}(f,g) \leq M\Phi_n(f,g). \]
If the pair $f$, $g$ 
is coincidence essentially reducible then 
\[ N\Phi_{n}(f,g) =\sum_{m|n} NP_m(f,g). \]
\end{theorem}
\begin{proof} The proof of 
homotopy invariance is analogous to the 
proof in \cite{hy} for periodic points, and involves considering what turn out to be
isomorphic systems of coincidence $n$-representatives. In fact, 
the only real difference is that
we deal here with 
classes, rather than with orbits.

For the second inequality, recall that $ M\Phi_n(f, g)$ is defined in terms of the
disjoint union of the $P_m(f,g)$. By homotopy invariance we need only show,
for an arbitrary pair $f$, $g$ 
that $\#(\bigsqcup_{m|n}P_m(f,g)) \geq N\Phi_{n}(f,g)$. We use
$\bigsqcup_{m|n}P_m(f,g)$ to define a set of $n$ representatives as follows. 
Let $x \in \bigsqcup_{m|n}P_m(f,g)$, then $x \in P_j(f,g)$ for some $j\mid n$. 
So $\rho_j([x]) \in {\cal R}(f^j_*, g_*^j)$. Let ${\cal S}$ be the set of all such 
$\rho_j([x])$ for all $x$ and for all $j$, and let $[\alpha]^k \in {\cal E}(f^k, g^k)$ 
be arbitrary, where $k\mid n$.
To show that ${\cal S}$ is a set of $n$-representatives, 
we must show that $[\alpha]^k$, reduces to some element of ${\cal S}$.
Since  $[\alpha]^k$ is essential, then 
$[\alpha]^k = \rho_k([x])$ for some 
$x \in \Phi(f^k, g^k)$. 
Let $m\mid k$ be the least positive integer such that $x \in \Phi(f^m, g^m)$, then
$x \in P_m(f,g)$.  It is not hard to see from the definitions  that
 $\rho_m([x])$ boosts to $[\alpha]^k$, even if $f$ and $g$ do
not commute (so we cannot use Lemma \ref{diagrammatic}).  We can of course 
use the same path to define the $\rho$ at levels
$m$ and $k$. Thus
${\cal S}$ is indeed a set of $n$-representatives and 
$$\#(\bigsqcup_{m|n}P_m(f,g)) \geq \#({\cal S}) \geq N\Phi_{n}(f,g),$$
as required. 

The first inequality follows since any set of $n$-representatives 
will contain the set of
all irreducible essential classes. In particular we will always have that 
$ \# \mathcal G \geq \sum_{k \mid n} NP_n(f,g)$ for any set of $n$-representatives
$\mathcal G$. Equality occurs exactly when $\mathcal G$ is
the set of all irreducible essential classes, and this happens when the pair $f$, $g$ 
is coincidence essentially reducible
\end{proof}
\begin{remark}  We want to make a comment about what could happen in this
proof if we  defined 
$M\Phi_n(f, g)$ in terms of the ordinary rather than the 
disjoint union.  The main point is that because,  as in  Example 
\ref{nondivide}, the 
$P_m$ need not be disjoint, nor can we  (automatically) assume that 
coincidence classes are not singletons, so then neither can we  deduce that 
 $\#(\bigcup_{m|n}P_m(f,g)) \geq \#({\cal S})$.  We will come back to this point
 in section \ref{openquestions} where we discuss related  open questions.

By analogy with Proposition \ref{pfg=pf}, we compare 
 $N\Phi_n(f, \id)$ and $N\Phi_n(f)$.
\begin{proposition} If $f: X \to X$ is essentially toral, then $N\Phi_n(f,\id) = N\Phi_n(f)$. \hfill 
${\Box}$
\end{proposition} \end{remark}

\begin{theorem}\label{toral2} We have $N\Phi_{n}(f,g)\geq N(f^m,g^m)$ for all $m\mid n$.
Moreover if $f,g$ are coincidence essentially reducible to the gcd at level $n$ with
$f^n$ and $g^n$ coincidence weakly
Jiang,  $N(f^n,g^n) \neq 0$ and the 
$\iota_{m,n}$  injective on essential boosts for all $m\mid n$, then
$$N\Phi_{n}(f,g)= N(f^n,g^n).$$
\end{theorem}
\begin{proof}
The first part is easy, since any set of $n$-representatives must contain at least one class to which each of the
essential class of ${\cal R}(f^n, \ g^n)$ reduce. 
Apart from the fact that we are working with classes rather than orbits, 
the second part of the proof works
without modification from periodic point theory (see \cite{hk1, hk2}). 
\end{proof}
\begin{examples} \label{KB3+} 
Continuing example \ref{KB2} on the Klein Bottle, we have that 
$N\Phi_{6}(f,g) = N(f^6,\ g^6)$ $ = 10859184$. Continuing Example 
\ref{matrixexample1} on the 2 Torus, we have that $N\Phi_{30}(F,G)$ 
$= N(F^{30},G^{30})$ $ = 221073919719792987930625. $
\end{examples}

The next  result generalizes its analogue in \cite{hy}, 
and shows what happens, in  Theorem \ref{toral2},  when we allow 
$N(f^n,g^n)$ to be equal to zero. 
Its proof contains nothing new and is omitted. 
For given $f, g :X\rightarrow X$, and a fixed natural number $n$, we define the set 
$M(f,g,n)$ as the set of $m\mid n$ with $N(f^m, g^m)\neq 0$. 
\begin{theorem} 
\label{tree} 
Suppose for a fixed positive integer $n$ we have that for each 
$m\in M(f,g,n)$ the pair $f,g$ is essentially reducible to the gcd at level $m$,
that the $f^m,g^m$ are weakly
Jiang, and furthermore that the 
$\iota_{q,m}$ are injective on essential boosts for each $q\mid m$. 
Then 
\[ 
N\Phi_n(f)=\sum_{\emptyset \neq \mu \subseteq M(f,g, n)} 
(-1)^{\# \mu -1}N(f^{\xi (\mu )}), 
\] 
where $\xi (\mu )$ is the gcd
 of the elements of $\mu$. 
Furthermore, M\"obius inversion can be used 
to obtain the $NP_q(f)$ for each $q \mid m$ for $m \in M(f,g, n)$. 
\end{theorem} 

\section{Connections with Periodic point theory when $g$ is invertible.}\label{invertiblesection}
This section should be thought of as an extended remark, rather than a section where rigorous 
proofs are given. Our intention is simply to inform intuition.  We start with the fact that
when $g$ is invertible and $f$ and $g$ commute, the set of coincidences of iterates for the pair $f,g$ is exactly 
the same as the set of periodic points of the map $g^{-1}f$.
 If in addition $X$ is essentially toral (where in the context of periodic point theory, counting points by depth of orbit has no advantage), then there is an 
\lq \lq isomorphism" of the coincidence Nielsen theory of iterates given here, 
and the Nielsen periodic
point theory of $g^{-1}f$. With the exception of Example \ref{projsp} which is given in this 
section for the purpose of illustration, 
all examples given in this paper are on spaces that in the periodic point sense
are essentially toral. 
\subsection{Relationships for invertible $g$, 
with the periodic point numbers $NP_n(g^{-1}f)$ and $N\Phi_n(g^{-1}f)$.}
If we specifically assume that $g$ is invertible, and that $f$ and $g$ commute, 
then the coincidence
points of $f^n$ and $g^n$ are exactly the same as the fixed points of
$(g^{-1}f)^n$. To see this, suppose that $g^n(x)= f^n(x)$, then (composing both sides with
$g^{-n}$) we have $x = g^{-n}(f^n(x)) =
(g^{-1} f)^n(x)$. This of course is reversible. 
Notice that this requires the commutativity of $f$ and $g$ as maps. 
On the other hand, we need only the invertibility of
$g_*$, together with the commutativity of $f_*$ and $g_*$, to effect a
well defined 
correspondence between $\Reid(f^n_*,g^n_*)$ and $\Reid((g^{-1}_*f_*)^n)$. In fact
this one to one correspondence goes
much deeper, as we will now state, but do not prove (we use the notation of \cite{heathfibresurvey}). 
\begin{theorem} If $g_*$ is invertible, and 
$f_*$ and $g_*$ commute at the $\pi_1$ level, then for all $m\mid n$ the 
homomorphisms $g^{-m}_*: \pi_1(X) \to \pi_1(X)$ induce well-defined bijections
$g^{-m}_*: \Reid(f_*^m, g_*^m) \to \Reid((g_*^{-1}f)^m)$, and the left hand diagram below commutes. If $g$ itself is invertible, and 
$f$ and $g$ commute as maps, then the equality in the right hand diagram exists, and the diagram commutes.
\[ 
\begin{CD}
\Reid(f_*^m,g_*^m) @>\iota^{f_*,g_*}_{m,n}>> \Reid(f_*^n,g_*^n) \\
@Vg_*^{-m}VV @VVg_*^{-n}V \\
\Reid((g_*^{-1}f_*)^m) @>\iota^{g_*^{-1}f_*,1}_{m,n}>> \Reid((g_*^{-1}f_*)^n)
\end{CD} \ \ \ \ 
\begin{CD}
\Phi(f^n,g^n)/\sim @>\rho_n>> \Reid(f_*^n, g_*^n) \\
@| @VVg_*^{-n}V \\
\Phi((g^{-1}f)^n)/\sim @>\rho_n>> \Reid((g_*^{-1}f_*)^n) 
\end{CD} \quad
\]
where notation for the boosts are the product versions of those in the proof of Theorem \ref{gcdmatrix}. \hfill ${\Box}$
\end{theorem}

In order to avoid getting into too many technical considerations, we make the following definition.
\begin{definition} Let $f_*$ and $g_*$ commute and 
$g_*$ be invertible. We say that the function 
$g^{-m}_*: \Reid(f_*^m, g_*^m) \to \Reid((g_*^{-1}f)^m)$ is \emph{essentiality preserving at level 
$m$} if a class 
$[\alpha]^m \in \Reid(f_*^m, g_*^m)$ is essential if and only if the class
$g^{-m}_*([\alpha]^m)$ is essential in $\Reid((g_*^{-1}f)^m)$. 
\end{definition} 

In the context of tori, 
when $N(f,g)$ or $N(g^{-1}f) \neq 0$ (see below), then all 
fixed or coincidence classes of linearized maps 
are singletons and essential. So then all the examples in this paper on tori are easily seen to be essentiality preserving at every level. 
It is a bit more work to see that 
the definition is satisfied for pairs of commuting maps on nil and solvmanifolds. The 
product theorem for semi-index (\cite{jez92}), together with the usual product formula for the
fixed point index can be used in this regard.
The definition holds in Example \ref{projsp} for a different reason. 

The main result which we state, but do not prove, is the following:
\begin{theorem} \label{isonthys} Let $f$ and $g$ be such that $f_*$ and $g_*$ commute,
with 
$g$ invertible, and suppose that 
the homotopy class of $f$ and $g$ contain a pair of commuting maps. 
If $g_*^{-m}$ is essentiality preserving for all $m\mid n$ and $g^{-1}f$ is 
essentially toral, then 
$$ \ \ \  \ \ \ \ \ \ \ \ \ \ \ \ \ \ \ \ \
NP_n( f, g ) = NP_n(g^{-1}f) \ \ \ \mbox{ and } N\Phi_n( f, g ) = N\Phi_n(g^{-1}f). 
\ \ \  \ \ \ \ \ \ \ \ \ \ \ \ \ \ \ \
{\Box} $$
\end{theorem}
We cannot rescind the hypothesis of essential torality (though it can be weakened). To see this, 
we can use Jiang's classical Example on $\R P^3$ (\cite[Example 4 p.67]{bojujiang}), where of course $g$ is the identity. Jiang's example illustrates why we need to use orbits
in periodic point theory.
The equalities in the Theorem
can be observed in practice with Example \ref{matrixexample1}, by using,
in the case of the $NP_n$, the periodic point 
M\"obius inversion formula for tori (\cite[Theorem 1.2]{hk1}). In the case 
of the $N\Phi_n$ we need the analogue of the last part of 
Theorem \ref{toral2} (also \cite[Theorem 1.2]{hk1}). For both numbers we also need 
the following observation for invertible $G$:
\begin{align*}
N(F^n, G^n) & = |\det(G^n - F^n)| \\
& = |\det(G^{n})||\det(I - G^{-n} F^n)| \\
&= |\det(I -G^{-n}F^n)| = N(G^{-n}F^n).
\end{align*} 
The penultimate step follows since $|\det(G^{n})|= 1$. This is
because when $G$ is invertible over ${\Z}$, we must have that $\det(G) = \pm 1$. 
\begin{remark}\label{weckenremark2} The result of Theorem \ref{isonthys} that for invertible $g$ we have that 
$NP_n( f, g ) = NP_n(g^{-1}f)$ and $N\Phi_n( f, g ) = N\Phi_n(g^{-1}f)$ when
the spaces are essentially toral, gives a strong indication
that we have not really lost anything by not being able to work with orbits on these spaces. 
The earlier comments about classes being singletons also confirms this, for the 
advantage of orbits only works when there are more than one non-removable 
points in 
each Nielsen class. We will not explore this here, but these considerations are 
related to the question of these spaces being Wecken, that is when the
homotopy classes of our maps 
contain representatives that attain the given lower bound. As the work of
You in periodic point theory demonstrates (\cite{you1}), 
in the presence of Nielsen numbers that are zero, 
the proofs get complicated even on tori. 
\end{remark}
\subsection{Invertible $g$, and orbits }\label{orbitsection}
When $g$ is invertible, and the pair $f$ and $g$ contain a pair of commuting maps
within their homotopy classes, then there is a way to define orbits. We simplify this 
subsection by making the blanket assumption that the commuting pair has been chosen,
leaving the subtler details of the theory as given in section three, to the reader.
In fact we will leave a great deal to the reader, 
dealing only with the $NP$ number. Consistent with the stated goal of this section we 
only include enough to inform intuition. 
The definition and lemma below are 
intended to indicate that the invertibility 
of $g$ allows us to define orbits. In particular the Lemma
gives that part of \cite[Proposition 1.14]{hpy} that was missing from Lemma 
\ref{diagrammatic} in Section \ref{iteratesection}. 
\begin{definition} Let $f, \ g$ be a pair of commuting self maps, with $g$ invertible as a 
map, and let $x \in \Phi(f^n, \ g^n)$. Then the {\it geometric orbit} of $x$ 
is the set 
\[ \{x, g^{-1}f(x), g^{-2}f^2(x), \dots, g^{-n+1}f^{n-1} (x)\} \] 
\end{definition}
If $x \in P_n(f, \ g)$ then the elements of the above list are all distinct, and clearly 
since $g^{n}(x) = f^{n}(x)$, then $g^{-n}f^{n} (x) = x$. The other analogous definitions
of periodic
point theory are forthcoming. 
\begin{lemma} (c.f. \cite[Proposition 1.14]{hpy})
Under the conditions of the definition the function $ g^{-1}f$ induces well defined 
functions $ g^{-1}f: \Phi(f^n,g^n) \to \Phi(f^n,g^n)$. This function respects 
both the Nielsen and Reidemeister relationships, and induces 
essentiality preserving functions (denoted $g^{-1}f$ and $(g^{-1}f)_*$ respectively)
on the respective sets of classes. 
Moreover, the 
following diagram is commutative
\[ \begin{CD}
\Phi(f^n,g^n)\sim @> \rho_n >> \Reid(f_*^n,g_*^n) \\
@V g^{-1}f VV @VV (g^{-1}f)_*V \\
\Phi((g_{-1}f)^n)\sim @>  \rho_n >> \Reid((g_{-1}f)_*^n),
\end{CD} 
\]
so that both geometric and algebraic orbits are well defined. Thus the 
notion of irreducible and essential orbit is well defined.  \hfill ${\Box}$
\end{lemma} 
\begin{definition}\label{inv} Let $f, \ g$ be a pair of self maps whose homotopy classes 
contain a commuting pair, and suppose that $g$ is invertible. We define 
$NP^{\cal INV}_n(f, g)$ to be  $n$ times the number of 
irreducible essential periodic point orbits of the map $g^{-1}f$. That is
$NP^{\cal INV}_n(f, g):= NP_n(g^{-1}f)$.
\end{definition}
\begin{theorem}\label{nplowerboundinv} Under the conditions of Definition \ref{inv} we 
have that 
$NP^{\cal INV}_{n}(f,g)$ is homotopy invariant in $f$ and $g$, and 
\[  \ \ \ \ \ \ \ \ \ \ \ \  \ \ \ \ \ \ \ \ \ \ \ \ \ \ \ \ \ \ \ \ \ \ \ \
NP^{\cal INV}_{n}(f,g) \le MP_{n}(f,g) \le \#(P_{n}(f,g)). 
 \ \ \ \ \ \ \ \ \ \ \ \  \ \ \ \ \ \ \ \ \ \ \ \ \ \ \ \ \ \ \ \ \ \ \ \
{\Box}  \]
\end{theorem}
Note that we no longer require essential torality in the following analogue of
Theorem \ref{isonthys}. 
\begin{theorem} \label{isonthysinv}
Under the conditions of Definition \ref{inv} we 
have that 
$$  \ \ \ \ \ \ \ \ \ \ \ \  \ \ \ \ \ \ \ \ \ \ \ \  \ \ \ \ \ \ \ \ \ \ \ \
NP^{\cal INV}_n( f, g ) = NP_n(g^{-1}f).
 \ \ \ \ \ \ \ \ \ \ \ \  \ \ \ \ \ \ \ \ \ \ \ \ \ \ \ \ \ \ \ \ \ \ \ \
{\Box} $$
\end{theorem}
In Theorem \ref{pseudomob} we showed that 
$N(f^n, \ g^n) \geq NP_n(f,g)$. But this does not generalize to the
$NP^{\cal INV}_n( f, g)$ numbers, as the following example shows:
\begin{example} \label{projsp} Let $\tilde f, \tilde g: S^2 \to S^2$
be maps of degree 3 and -1 respectively. We can think of them as the respective suspensions 
of the same degree maps on $S^1$. In this way  $\tilde f$ and $\tilde g$  are seen to 
be ${\Z}_2$  equivariant maps that induce self maps $f$ and $g$ respectively on
Real Projective space 
${\R}P^2$.  It follows easily that  $f$ and $g$ induce  identity 
homomorphisms $f^n_*$ and $g^n_*$ on $\pi_1({\R}P^2) \cong {\Z}_2$, 
 for all positive integers $n$.  Jezierski in \cite[Corollary 5.1]{jg}) has worked out 
$N(f,g)$ in detail for all pairs of self maps $f,g$ of ${\R}P^2$. In particular for our $f$ and $g$
we have that $N(f^n, g^n)= 2$ for all positive integers $n$. 
Furthermore since 
 $g$ is invertible, $f$ and $g$ commute, and  
$(g^{-1}f)_*$ is the identity,  then each
periodic point orbit of $g^{-1}f$  has length $1$ at every level.
Now let $n= 2^r$ for some positive integer $r$. Since for 
any $m\mid n$ the number $n/m$ must be
even, it is not hard to see 
that $\iota_{m,n}$ is
multiplication by an even integer, that is, it is the zero function
on ${\Z}_2$. In particular the orbit
$\langle [1]^n \rangle$ is irreducible and essential. From above
$$NP^{\cal INV}_n(f, g) = NP_n( g^{-1}f) = n = 2^r > N(f^n, g^n) = 2.$$
\end{example}
\section{Roots of iterates and the work of Brown,  Jiang and  Schirmer}\label{rootsection}
As mentioned in the introduction, Brown Jiang and Shirmer (\cite{BJS}) have already
studied  roots of iterates in the context of Nielsen theory. 
One could be forgiven for assuming 
that putting $g$ equal to a constant map in our theory would give
many of the results of \cite{BJS}. 
Actually there is very little intersection between the two theories.  In fact they are 
profoundly incompatible in at least two important ways. The first has 
to do  with the incompatibility of the work in 
\cite{BJS} and the periodic point theory we are seeking to generalize. The second
incompatibilility is reflected in the fact that  the root 
theory one obtains by putting $g$ equal to a constant map in our theory 
is homotopy invariant with respect to such $g$,  and the work in \cite{BJS} is not.  We 
illustrate these differences with the following example, which uses the same map
$f$ used  in
\cite[ Example 6.1]{BJS}.
\begin{example}
Let $f,a :S^1 \to S^1$ be maps with $f(z)=z^2$ and $a$ the constant map $a(z)=1$. Then $N(f^n,a^n) = 2^n$. When $p$ is a prime, $n=p^k$ where $k\neq 0$ is a positive integer, 
then any reducible class at level $p^k$ in our theory
 will be in the image of $\iota_{p^{k-1},p^k}$. 
Thus there are $2^{p^{k-1}}$ reducible classes at level $n$. So then 
\[ 
NP_{p^k}(f,a) = 2^{p^k} - 2^{p^{k-1}},
\]
and by Theorem \ref{nphilowerbound} we have
\begin{align*}
N\Phi_{p^k}(f,a) &= \sum_{i=0}^k NP_{p^i}(f,a) = 2 + \sum_{i=1}^k NP_{p^i}(f,a) 
= 2 + \sum_{i=1}^k 2^{p^i} - 2^{p^{i-1}} =2^{p^k}.
\end{align*}

In \cite{BJS}, the authors define a Nielsen type number denoted $ NI_n(f,a)$, which they
call the \emph{Nielsen number of irreducible roots at level $n$ of $f$ at $a$}. 
As we will see below, while this sounds like our $NP_n(f,a)$,  it is not.  We say more below,
but we give neither the definition, nor details of thier computations. 

With  
the very same $f$ and $a$,  the number $ NI_n(f,a)$ is computed in 
\cite[Example 6.1]{BJS} to be 
\[ NI_n(f,a) = \begin{cases} 2 & \text{ for $n=1$}, \\
2^{n-1} & \text{ for $n>1$}. \end{cases} \]
If $n$ is prime, then $NP_n(f,a) = 2^n - 2$, and $N\Phi_n(f,a) = 2^n$, 
while $NI_n(f,a) = 2^{n-1}$, so $ NI_n(f,a)$ is different from both $NP_n(f,a)$ and $N\Phi_n(f,a)$. 
\end{example}

\begin{remark}{\bf Homotopy invariance incompatibility. }
As we mentioned above, the number  $ NI_n(f,a)$ depends on the  variable \lq \lq $a$". 
In this regard, 
we were careful to state in the last part of the example above, that the 
computation of  $NI_n(f,a)$ from \cite{BJS} was with respect to the very same 
$f$ and $a$ as in the first part of the example.  Since $S^1$ is path connected,
any two constant maps are homotopic.  In our theory this means that 
the root theory obtained by putting
$g$  equal to  constant map at $a$ is independent of the choice of 
$a \in S^1$.  This is not the case with the root theory in \cite{BJS}.  In \cite{BJS}, 
the number  $NI_n(f,a)$ is dependent on the period of the chosen $a$ under $f$ (which can 
be either finite or infinite).   In particular in the very 
same example in \cite{BJS}, when $a(z)=-1$ (so $a$ has period 2),
 then with the same  $f$,  we have that $NI_n(f,a) = 2^{n}$. 
\end{remark}

\begin{remark}{\bf Incompatible definitions of \lq \lq boosting". }
A second   way the two theories are incompatible is to be found in the very 
definition of
reducible classes. Though we use the same words to describe
$ NI_n(f,a)$ and $NP_n(f,a)$ (each is 
defined to be 
the number of irreducible essential classes), the two theories are very different because these
words mean different things in the two papers.  The main point is that 
 in \cite{BJS} there are  
\lq \lq boosts" from levels $m$ to $n$ for certain $m$ which do not divide $n$. 

Consider a hypothetical situation where we have 
an $f$ with $f(a) = a$ for some $a \in X$. As above, by abuse of notation we use $a$ to denote the function with 
$a(x) = a$ for all $x$. Note that whenever $f^m(x)=a$ then
$f^{m + 1}(x) = f(f^{m}(x)) = f(a) = a$ too. Suppose, for
example,  that $x$ is a root of
$f$ at $a$ of least period $4$. Then $f^4(x) = a$ but $f^j(x) \neq a$ for 
any $j< 4$. But  $f^5(x)= f(a)= a$, and of course $4 \not | \  5$. 
Theorem 2.1 of \cite{BJS} (about which we will say more below)
implies that this kind of boosting
extends to Nielsen classes in a natural way. So then in \cite{BJS} the class at level 
5 could be  geometrically reducible for Brown Jiang and Schirmer,
 but  geometrically irreducible in our theory. 

The algebraic side of this type of reducibility for constant maps $a$ 
can also  work. 
To see this, note that $a_*$ is the trivial homomorphism. In this 
case, Definition \ref{boostingfunctions} simply becomes 
$\iota_{m,n}(\alpha) = f_*^{n-m}(\alpha)$, 
and this can be well defined on Reidemeister classes 
even when $m$ does not divide $n$. 
\end{remark}

We wish to say one more thing  about incompatibility of the two 
theories. It pertains to the difficulty  of 
generalizing the theory of roots of iterates of the type  given in 
\cite{BJS},  to coincidences.  
  The problems start with the fundamental result of \cite{BJS} (Theorem 2.1) upon which
their whole paper is based. That Theorem  states that if $R^m$ and $R^n$ are root classes with 
$m<n$, and  if $R^m \cap R^n \neq \emptyset$, then $R^m \subseteq  R^n$. 
Example \ref{nondivide} provides a ready made counter example to a 
 coincidence  generalization of this phenomenon. Recall that
 we constructed a
$q \in P_3(f,g) \cap P_5(f,g)$.  We  contend 
that $0$ and $q$ are Nielsen equivalent at level 3, but not at level 5. To see this let 
$c$ be the linear path from $0$ to $q$. Since neither of the paths
$g^3(c)$ nor $f^3(c)$ traverse $\bar 0$, we have that $g^3(c) \sim f^3(c)$ at level 3. 
On the other hand while $g^5(c)$ does not
traverse $\bar 0$, the path $f^5(c)$ does (in fact $f^5(c)$ is a generator of $\pi_1(S^1)$), 
and so $0$ and $q$ cannot be Nielsen equivalent at level 5.

\section{Open Questions}\label{openquestions}
We conclude the paper with two open questions. The first concerns the
definition of
$M\Phi_n(f,g)$, the second is basically the Wecken question. We comment on
both.
\newtheorem{question}[tpcl]{Open Question}
\begin{question}
Is the set
$$M\Phi_n(f,g) = \min \# \left\{\bigcup_{m|n} P_{m}(f_1,g_1) | f_1 \sim f \mbox{
and } g_1 \sim g\right\}.$$
on manifolds?
\end{question}
The question we are asking is if maps that actually attain the minimum
number at all levels simultaneously would ever
contain \lq \lq stray" coincidences of iterates (like $q$ in Example
\ref{nondivide}) that lie in the intersection of different $P_m(f,g)$.
Our feeling (guess)
is that such points are rare, and that
they do not form essential singleton classes. In other words it is our
feeling that
(again like $q$ in Example \ref{nondivide}) such stay points lie in
Nielsen classes that are
either not essential, or not singletons. In either case we should be able
by some homotopy,
either to
remove them, or move them even slightly,
so that they no longer belong to more than one
of the $P_m(f,g)$.
\begin{question} \label{question2}
Under what conditions are the numbers $NP_n(f,g)$ and
$N\Phi_n(f,g)$ Wecken, so that there exist maps
$ f_1 \sim f \mbox{ and } g_1 \sim g$ such that
$NP_n(f,g)= MP_n(f,g)$ and $N\Phi_n(f,g)= M\Phi_n(f,g)$?
\end{question}
Question \ref{question2} should probably be two questions, one for each
number.
The following example shows that such conditions do exist.
\begin{example}
Let $f, g: T^2 \to T^2$ be maps of the $2$ torus with linearizations $F$
and $G$ respectively, where $F$ and $G$ are as given in Example
\ref{matrixexample1}. Then $f$ and $g$ satisfy the
hypotheses of the last part of Theorem \ref{toral2}, and so
$$N\Phi_n(f^n,g^n) = |\det(F^n - G^n)|= \#(\Phi(F^n,G^n)) =
M\Phi_n(f,g).$$
\end{example}
The second equality holds, since when in situations like this $\det(F^n -
G^n) \neq 0$,
the linear representations of $f$ and $g$ have that all classes are
singletons.
The last part follows, since $N\Phi_n(f^n,g^n) \leq M\Phi_n(f,g) \leq \#(\Phi(F^n,G^n))$.
Part of what makes the example work is that when $fg=gf$, then
$$\Phi (f^n, g^n) = \bigcup_{m \mid n} \Phi(f^m, g^m) =\bigcup_{m \mid n}
P_m(f, g)
\neq \bigsqcup_{m \mid n} P_m(f, g).$$
But this is hard to
generalize. In particular the analogue of the
technique
used in periodic point Wecken theorems (see \cite{jez06}), which produces
periodic points in
empty classes that boost to essential classes at a higher level, will not
work unless the
modified maps commute.

\section*{Appendix: Reduction to the GCD and the circle}
In this section we determine exactly which pairs of maps on the circle are essentially reducible to the gcd. As in the previous sections, we identify a map $f$ on the circle with its integer degree. Our goal is the following theorem:
\begin{theorem}\label{circlegcd}
Given integers $a$ and $b$, the maps $f=a$ and $g=b$ are essentially
reducible to the gcd if and only if $\gcd(a,b) = 1$.
\end{theorem}

Though Theorem \ref{gcdmatrix} is, from one point of view, 
more general than the above theorem
in that it includes all tori, it is, from another point of view, less general in 
that it requires the map $g$ to be invertible. The proof of Theorem \ref{gcdmatrix}, 
which involved factoring certain polynomials, does not generalize to the setting where
$g_*$ is not invertible. So then we need to give a separate proof of this theorem.

Our proof relies heavily on a lemma concerning the factors of the polynomials
 $\iota_{p,q}^{x,1}$ given in the proof of Theorem \ref{gcdmatrix}.
 Since we want to emphasize 
in this proof that the $\iota_{p,q}^{x,1}$ are polynomials in $x$, 
we change notation and denote
$\iota_{p,q}^{x,1}$ by $\sigma_{p,q}(x)$. So then 
\[ \sigma_{p,q}(x) = 1 + x^p + x^{2p} + \dots + x_,^{(q-1)p} \mbox{ and } 
\sigma_{1,q}(x) = 1 + x + \dots + x^{q-1} \]
and both are polynomials in $\Z[x].$
As in Theorem \ref{gcdmatrix}
we can, without loss of generality consider only the case 
when the gcd of the levels to which we consider reductions, is 1. With this in mind, 
our factoring result is as follows:
\begin{lemma}\label{cyclolemma}
When $\gcd(k,m)=1$ with $n=km$, for the above polynomials
$\sigma_{p,q}(x) \in \Z[x]$ there is a polynomial $p(x) \in \Z[x]$ satisfying:
\begin{align} 
\sigma_{k,n}(x) &= p(x) \sigma_{1,m}(x) \label{kndiv} \\
\sigma_{m,n}(x) &= p(x) \sigma_{1,k}(x) \label{mndiv} \\
\sigma_{1,n}(x) &= p(x) \sigma_{1,k}(x) \sigma_{1,m}(x) \label{1ndiv}
\end{align}
\end{lemma}
\begin{proof}
Our approach is to fully factor $\sigma_{1,k}$, $\sigma_{m,n}$
and $\sigma_{k,n}$ in $\Z[x]$, and compare the factors.  Consider the two
factorizations of $x^k -1$:
\[ x^k - 1 = \prod_{d\mid k} \Phi_d(x) = (1 - x)\sigma_{1,k}(x), \]
where the $\Phi_d(x)$ are the cyclotomic polynomials (see
\cite{lang}). Cancelling $x-1$ we have that $ \sigma_{1,k}(x)$ is the unique product
\begin{equation}\label{1kfactors} \sigma_{1,k}(x) = \prod_{d \mid k,\, d \neq 1} \Phi_d(x). \end{equation}
of distinct (non-repeating) irreducible polynomials over the UFD $\Z[x]$.

We next factorize $\sigma_{k,n}$ as follows:
\[ \sigma_{k,n}(x) = 1 + x^k + x^{2k} + \dots + x^{(m-1)k} =
\sigma_{1,m}(x^k) = \prod_{c \mid m,\, c \neq 1}
\Phi_c(x^k). \]

Lemma 1 of \cite{cmw95} states that when $\gcd(c,k) = 1$, we have
\[ \Phi_c(x^k) = \prod_{\lcm(r,k)=ck} \Phi_r(x), \]
so then  our complete factorization of $\sigma_{k,n}(x)$ is:
\[
\sigma_{k,n}(x) = \prod_{c\mid m,\, c \neq 1} \prod_{\lcm(r,k)=ck}
\Phi_r(x).
\]

We need an alternative characterization of this product: So let
%$R$ be the set of $r$ such that there is some nontrivial $c \mid m$
%with $\lcm(r,k) = ck$. We will argue that this is the same as the set $S$
%of numbers of the form $s = ab$ where $a$ and $b$ are factors of $m$
%and $k$, respectively, and $a \neq 1$. 
%
$$\mbox{} R =\{r| \mbox{ there exists } c \neq 1 \mbox { with } c \mid m \mbox{ and } \lcm(r,k) = ck \}$$ $$ \mbox{ and let }
S =\{ ab | \ \ a|m \mbox{ and } b|k \mbox{ and } a \neq 1 \}.$$
We claim that $R = S$.
To show that $R \subseteq S$, let $r \in R$ together with a chosen $c\mid m$ with 
$ck = \lcm(r,k)$ Since $ \lcm(r,k) = rk/\gcd(r,k)$, then
$c = r/\gcd(r,k)$ and so $r = c \gcd(r,k)$. 
So if we let $a= c$, then $a \mid m$ and if $b=\gcd(r,k) \mid k$,
then since $c$ is assumed to be nontrivial, 
we have $a \neq 1$, and thus
$r = ab\in S$.

Now we show  $S \subseteq R$: Let $r \in S$ with
$r = ab$ with $a$ and $b$ factors of $m$ and $k$ respectively such that 
$a\neq 1$. Let $c = a$, and then we automatically have $c \mid m$ and
$c \neq 1$. It remains to show $\lcm(r,k)=ck$. Since $k$ and $m$ have no
common divisors, we have $\gcd(r,k) = b$. Then 
\[ \lcm(r,k) = \frac{rk}{\gcd(r,k)} = \frac{rk}{b} = \frac{r}{b}k = ak
= ck \]
as desired. 

This now allows us to write $\sigma_{k,n}(x)$ as:
\[ \sigma_{k,n}(x) = \mathop{\prod_{r =ab}}_{a \mid m, b
\mid k, a \neq 1} \Phi_r(x). \ \ \mbox{ Similarly } 
\sigma_{m,n}(x) = \mathop{\prod_{r =ab}}_{a \mid m, b
\mid k, b \neq 1} \Phi_r(x).
\]

By uniqueness of the factorizations above it is clear that any factor of $\sigma_{1,k}(x)$ appearing
in \eqref{1kfactors},  is also a factor of $\sigma_{m,n}(x)$ (one in
which $a = 1$). Thus we
have $\sigma_{1,k} \mid \sigma_{m,n}$, with quotient
\[ p(x) = \sigma_{m,n}(x) / \sigma_{1,k}(x) =
\mathop{\mathop{\prod_{r=ab}}_{a\mid k, b\mid m}}_{a \neq 1, b \neq
1} \Phi_r(x). \]

We have established \eqref{kndiv}, and by symmetry that 
\[ \sigma_{k,n}(x) / \sigma_{1,m}(x) = p(x) \]
This establishes \eqref{mndiv}. 

For statement (4), observe that 
\[ \sigma_{1,n}(x) = \prod_{r \mid n, r \neq 1} \Phi_r(x) =
\mathop{\prod_{r = ab}}_{a \mid k, b \mid m} \Phi_r(x). \]
As above,  we see that any factor of $\sigma_{1,k}(x)$ appearing
in \eqref{1kfactors},  is a factor in this last product (one in which $b= 1$). 
Similarly any factor of $\sigma_{1,m}(x)$ is a factor in this last product in which $a = 1$. 
Thus
\[ \frac{\sigma_{1,n}(x)} {\sigma_{1,k}(x) \sigma_{1,m}(x)} = p(x), \]
and this establishes \eqref{1ndiv}.
\end{proof}

Lemma \ref{cyclolemma} is interesting in its own right because it provides a strategy for factoring the algebraic boosts in periodic points theory (in that setting, the algebraic boost $\Reid(f^p) \to \Reid(f^q)$ is exactly $\sigma_{p,q}(f)$).  
As with maps $f$ and $g$ of $S^1$,  we can 
identify the $\iota_{p,q}$ with  multiplication by some integer, 
and for the rest of the section we identify each
$\iota_{p,q}$ with its corresponding integer.  
We do as follows:
Let $f=a$ and $g=b$ then 
\[ \iota_{p,q} = a^{q-p} \sigma_{p,q}(b/a). \]
Note that, even though we evaluate $\sigma_{p,q}$ at a non-integer, 
multiplication by $a^{q-p}$ causes the result to be integral. Applying all this evaluation to 
the above lemma gives the following factorization result for the $\iota_{r,s}$: 
\begin{lemma}\label{iotalemma}
When $\gcd(k,m)=1$ with $n=km$, 
there is $p \in \Z$ satisfying:
\begin{align*} 
\iota_{k,n} &= p \iota_{1,m} \\
\iota_{m,n} &= p \iota_{1,k} \\
\iota_{1,n} &= p \iota_{1,k} \iota_{1,m}
\end{align*}
\end{lemma}

We use these factorizations to prove the following lemma.
\begin{lemma}\label{equivalent}
Let $n=km$ with $\gcd(k,m) = 1$.
Given integers $a$ and $b$ and maps
$f=a$ and $g=b$, the following are equivalent:
\begin{enumerate} 
\item \label{reds}  Two arbitrary classes $[w]^m$ and $[z]^k$ at levels $m$ and $k$ 
respectively, which boost to the 
same class $[\iota_{m, n}w]^{n} = [\iota_{k, n}z]^{mn}$ reduce to level 1. 
\item \label{lcms} $\lcm(\iota_{m,n}, \iota_{k,n})= \iota_{1,n}$.
\item \label{gcds} $\gcd(\iota_{1,k},\iota_{1,m})=1$.
\end{enumerate}
\end{lemma}
\begin{proof} Note that the first condition of the lemma is equivalent to saying that
$\im(\iota_{m,n}) \cap \im(\iota_{k,n})= \im(\iota_{1,n})$.  
Now the $\iota$ are injective, and 
$\im(\iota_{m,n})$ is a subgroup of ${\cal R}(a^n,b^n)$ of order $|b^m - a^m|$,
 generated by the integer $\iota_{m,n}$. Similarly
$\im(\iota_{k,n})$ is the subgroup of ${\cal R}(a^n,b^n)$ of order $|b^k - a^k|$, 
generated by the integer $\iota_{k,n}$. 
So $\im(\iota_{m,n}) \cap \im(\iota_{k,n})$ is the subgroup of ${\cal R}(a^n,b^n)$ generated 
by the integer 
$\lcm(\iota_{m,n}, \iota_{k,n})$. 
On the other hand $\im(\iota_{1,n})$ is the subgroup of ${\cal R}(a^n,b^n)$ generated by 
$\iota_{1,n}$. The first equivalence follows. 

To prove the equivalence of the second two statements,  consider the second statement.
By Lemma \ref{iotalemma},  the integers 
$\iota_{m,n}$ and $\iota_{k,n}$ have a
common factor $p$, with respective quotients $\iota_{1,k}$ and
$\iota_{1,m}$. Thus 
\[ \lcm(\iota_{m,n}, \iota_{k,n}) = p \lcm(\iota_{1,k},
\iota_{1,m}), \]
and the second statement is equivalent to $p \lcm(\iota_{1,k},\iota_{1,m}) =
\iota_{1,n}$. But Lemma \ref{iotalemma} also gives $\iota_{1,n} =
p\iota_{1,k}\iota_{1,m}$, and so the second statement is equivalent to $\lcm(\iota_{1,k}, \iota_{1,m}) =
\iota_{1,k}\iota_{1,m}$, which is to say that $\gcd(\iota_{1,k},
\iota_{1,m}) = 1$. 
\end{proof}

We are now ready to prove the main result of this section.
\begin{proof}[Proof of Theorem \ref{circlegcd}]
First we assume that $\gcd(a,b) \neq 1$, and show that $f$ and $g$ are not
essentially reducible to the gcd. If $a$ and $b$ have a nontrivial common
divisor, then $\iota_{1,m}$ 
and $\iota_{1,k}$ will share this divisor as well for any $m$ and
$k$. This is because each $\iota$ is a sum of terms, each of which is
divisible by $a$ or $b$. Thus we have
$\gcd(\iota_{1,k}, \iota_{1,m}) \neq 1$ for any $k$ and $m$. In
particular we can choose $k$ and $m$ to be relatively prime, and then
Lemma \ref{equivalent} will apply to show that $f$ and $g$ do not reduce to
the gcd at level $n=km$. 

For the converse, assume that $f$ and $g$ are not essentially
reducible to the gcd, and that $\gcd(a,b) = 1$. We deduce a contradiction.
    Since $\gcd(a,b) = 1$ if and only if $\gcd(a^d,b^d)= 1$ for all
positive integers $d$, then by the same argument used
in the proof of Theorem 
 \ref{gcdmatrix}, we may assume, without loss of generality, 
that the failure to reduce comes
at levels $k$ and $m$ with $\gcd(k,m)= 1$.  

By Lemma \ref{equivalent},  we have that $\iota_{1,k}$ and $\iota_{1,m}$
have a common prime factor $p$. It is easy to see that 
\[ (a-b)\iota_{1,k} = a^k - b^k, \text{ and } (a-b)\iota_{1,m} = a^m - b^m. \]
Thus $p$ divides $a^k - b^k$ and $a^m - b^m$, and so
\[ a^k = b^k \mod p, \text{ and } a^m = b^m \mod p. \]

Now since $\gcd(a,b) = 1$, the prime $p$ cannot divide both $a$
and $b$. Thus one of $a$ and $b$ is invertible mod $p$, without loss
of generality we assume that $b$ is invertible mod $p$. Then the above
can be written
\[ (a/b)^k = 1 \mod p, \text{ and } (a/b)^m = 1 \mod p. \]
Thus both $k$ and $m$ are divisible by the order of the element $a/b$
in the multiplicative group $\Z_p^*$. If this order is not 1, then
this will contradict the assumption that $\gcd(k,m) = 1$.

It remains to consider the case where the order of $a/b$ is 1, which
is to say that $a=b \mod p$. In this case, the definition of
$\iota_{1,k}$ simplifies modulo $p$:
\[ \iota_{1,k} = ka^{k-1} = kb^{k-1} \mod p, \]
and since $p \mid \iota_{1,k}$, we have $p \mid ka^{k-1}$ and $p \mid
kb^{k-1}$. Since $p$ is prime, either $p$ divides $k$, or $p$ divides
both of $a$ and $b$. But $\gcd(a,b) = 1$, and so we conclude that $p$
divides $k$. 
The same argument of the previous paragraph applied to $\iota_{1,m}$
shows that $p$ also divides $m$, which contradicts the assumption
that $\gcd(k,m)=1$.
\end{proof}

\end{document}